\documentclass[12pt]{article}
\usepackage{amsmath}
\usepackage{latexsym}
\usepackage{amssymb}
\newtheorem{thm}{Theorem}[section]
\newtheorem{la}[thm]{Lemma}
\newtheorem{Defn}[thm]{Definition}
\newtheorem{Remark}[thm]{Remark}
\newtheorem{Note}[thm]{Note}
\newtheorem{prop}[thm]{Proposition}

\newtheorem{Example}[thm]{Example}
\newtheorem{Examples}[thm]{Examples}
\newtheorem{Problems}[thm]{Problems}

\newtheorem{Problem}[thm]{Problem}
\newtheorem{Convention}[thm]{Convention}
\newtheorem{Number}[thm]{\!\!}

\newenvironment{example}{\begin{Example}\rm}{\end{Example}}

\newenvironment{rem}{\begin{Remark}\rm}{\end{Remark}}
\newenvironment{numba}{\begin{Number}\rm}{\end{Number}}
\newenvironment{proof}{{\noindent\bf Proof.}}%
                  {\nopagebreak\hspace*{\fill}$\Box$\medskip\medskip\par}   

\newcommand{\wb}{\overline}
\newcommand{\ve}{\varepsilon}
\newcommand{\at}{\symbol{'100}}

\newcommand{\mto}{\mapsto}

\newcommand{\N}{{\mathbb N}}
\newcommand{\R}{{\mathbb R}}

\newcommand{\F}{{\mathbb F}}

\newcommand{\K}{{\mathbb K}}

\newcommand{\Q}{{\mathbb Q}}

\newcommand{\Z}{{\mathbb Z}}

\newcommand{\cg}{{\mathfrak g}}

\newcommand{\dl}{{\displaystyle \lim_{\longrightarrow}}}

\newcommand{\sub}{\subseteq}

\DeclareMathOperator{\SL}{SL}

\DeclareMathOperator{\id}{id}

\DeclareMathOperator{\op}{op}

\DeclareMathOperator{\con}{con}
\DeclareMathOperator{\lev}{lev}
\DeclareMathOperator{\parb}{par}

\DeclareMathOperator{\ik}{ik}

\DeclareMathOperator{\car}{char}
\DeclareMathOperator{\PSL}{PSL}

\begin{document}
\begin{center}
{\Large\bf Contraction groups and the big cell for\\[2mm]
endomorphisms of
Lie groups over local fields}\\[7mm]
{\bf Helge Gl\"{o}ckner\footnote{Supported by Deutsche Forschungsgemeinschaft,
project GL 357/10-1.}}\vspace{4mm}
\end{center}
\begin{abstract}\vspace{.3mm}\noindent
If $G$ is a Lie group over a totally disconnected
local field and
$\alpha\colon G\to G$ an analytic endomorphism,
let $\con(\alpha)$ be the set
of all $x\in G$ such that $\alpha^n(x)\to e$
as $n\to\infty$.
Call sequence $(x_{-n})_{n\in\N_0}$ in~$G$
an \emph{$\alpha$-regressive trajectory} for~$x$
if $\alpha(x_{-n})=x_{-n+1}$ for all $n\in\N$
and $x_0=x$.
Let $\con^-(\alpha)$
be the set of all $x\in G$ admitting an $\alpha$-regressive
trajectory $(x_{-n})_{n\in\N_0}$ such that $x_{-n}\to e$
as $n\to\infty$.
Let $\lev(\alpha)$ be the set of
all $x\in G$ whose $\alpha$-orbit
is relatively compact,
and such that $x$ admits an $\alpha$-regressive
trajectory $(x_{-n})_{n\in\N_0}$
such that $\{x_{-n}\colon n\in\N_0\}$ is relatively compact.
The \emph{big cell}
associated to~$\alpha$ is the subset
$\Omega:=\con(\alpha)\lev(\alpha)\con^-(\alpha)$.
We show: $\Omega$ is open in $G$
and the product map
$\con(\alpha)\times\lev(\alpha)\times\con^-(\alpha)\to \Omega$,
$(x,y,z)\mto xyz$ is \'{e}tale
for suitable immersed Lie subgroup
structures on $\con(\alpha)$, $\lev(\alpha)$,
and $\con^-(\alpha)$.
Moreover, we study group-theoretic properties
of $\con(\alpha)$ and $\con^-(\alpha)$.\vspace{3mm}
\end{abstract}
{\footnotesize {\em Classification}:
22E20 (primary); 
22D05, 
22E25, 
22E35, 
32P05, 
37B05, 
37C05, 
37C86, 
37D10 
(secondary). \\[3mm]
{\em Key words}: endomorphism, automorphism,
Lie group, $p$-adic Lie group, local field, contraction group, anti-contraction group,
Levi subgroup, big cell, parabolic subgroup, nilpotent group, iterated kernel, HNN extension.}
\section{Introduction and statement of results}
Let $G$ be a totally disconnected,
locally compact topological
group, with neutral element~$e$,
and $\alpha\colon G\to G$ be a continuous
endomorphism.
Then various $\alpha$-invariant
subgroups of~$G$ can be associated
with $\alpha$, which are important for the
structure theory of totally disconnected,
locally compact groups and its applications
(see \cite{Wi3}, \cite{GBV}, \cite{BGT};
for automorphisms, see \cite{BaW}, \cite{GaW}, \cite{Sie};
cf.\ \cite{Wi1} and \cite{Wi2}
for the general structure theory).
According to~\cite{Wi3},
the \emph{contraction subgroup} of~$\alpha$
is the set $\con(\alpha)$
of all $x\in G$ such that $\alpha^n(x)\to e$
as $n\to\infty$.
The \emph{anti-contraction subgroup}
of~$\alpha$ is the set of all
$x\in G$ admitting an $\alpha$-regressive trajectory
$(x_{-n})_{n\in\N_0}$ (as recalled in the abstract)
such that $\lim_{n\to\infty}x_{-n}=e$.
Neither of the subgroups
$\con(\alpha)$ and $\con^-(\alpha)$ need to be closed in~$G$;
if $\alpha$ is a (bicontinuous) automorphism of $G$,
then $\con^-(\alpha)=\con(\alpha^{-1})$.
The \emph{parabolic subgroup} of $\alpha$
is the set $\parb(\alpha)$ of all $x\in G$
such that $\{\alpha^n(x)\colon n\in\N_0\}$
is relatively compact in~$G$;
the \emph{anti-parabolic subgroup}
of~$\alpha$ is the set $\parb^-(\alpha)$
of all $x\in G$
admitting an $\alpha$-regressive trajectory
$(x_{-n})_{n\in\N_0}$
such that $\{x_{-n}\colon n\in\N_0\}$
is relatively compact in~$G$.
Then $\parb(\alpha)$ and $\parb^-(\alpha)$
are closed subgroups of~$G$,
whence also the \emph{Levi subgroup}
$\lev(\alpha):=\parb(\alpha)\cap\parb^-(\alpha)$
is closed;
moreover,
\[
\alpha(\con(\alpha))\sub\con(\alpha),\quad
\alpha(\con^-(\alpha))=\con^-(\alpha),\quad
\alpha(\parb(\alpha))\sub\parb(\alpha),
\]
\[
\alpha(\parb^-(\alpha))=\parb^-(\alpha),\quad
\mbox{and}\quad
\alpha(\lev(\alpha))=\lev(\alpha);
\]
see \cite{Wi3} (notably Proposition~19).
We are interested in the subset
\[
\Omega:=\con(\alpha)\lev(\alpha)\con^-(\alpha)
\]
of~$G$, the so-called \emph{big cell}.
If $\con(\alpha)$ is closed
in~$G$, then also $\con^-(\alpha)$
is closed and the product map
\[
\pi\colon \con(\alpha)\times\lev(\alpha)\times\con^-(\alpha)\to\Omega
\]
is a homeomorphism, see \cite[Theorems D and F]{BGT}
(cf.\ already \cite{Wan}
for the case of automorphisms
of $p$-adic Lie groups).
If $\alpha$ is a $\K$-analytic endomorphism
of a Lie group~$G$ over a totally disconnected local
field~$\K$
and $\con(\alpha)$ is closed,
then $\con(\alpha)$, $\lev(\alpha)$,
and $\con^-(\alpha)$
are Lie subgroups of~$G$
and $\pi$ is a $\K$-analytic diffeomorphism,
see \cite[Theorem~8.15]{END}.
If $\K$ has characteristic~$0$,
then $\con(\alpha)$ is always closed
(see \cite[Corollary~6.7]{END}).
Our goal is twofold:
First, we strive for additional
information concerning the subgroups
just discussed, e.g.,
group-theoretic properties
like nilpotency.
Second, we wish to explore
what can be said when $\con(\alpha)$
fails to be closed.
While our focus is on Lie groups,
we start with a general observation:
\begin{prop}\label{fi-res}
Let $G$ be a totally disconnected,
locally compact group and $\alpha\colon G\to G$
be a continuous endomorphism.
Then $\Omega:=\con(\alpha)\lev(\alpha)\con^-(\alpha)$
is an open $e$-neighbourhood in~$G$
such that $\alpha(\Omega)\sub\Omega$
and the map
\[
\parb(\alpha)\times\parb^-(\alpha)\to\Omega,\quad (x,y)\mto xy
\]
is continuous, surjective, and open.
\end{prop}
Let $\K$ be a totally disconnected local
field.
For the study
of a $\K$-analytic endomorphism
$\alpha$
of a $\K$-analytic Lie group~$G$,
it is useful to consider
$(G,\alpha)$
as a $\K$-analytic dynamical
system with fixed point~$e$.
We shall use
\emph{local invariant manifolds}
for analytic dynamical systems
as introduced in \cite{INV}
and \cite{FIN},
stimulated by
classical constructions in the theory
of smooth dynamical systems over~$\R$
(cf.\ \cite{Irw} and \cite{Wel}).
See \cite{AUT} and \cite{END}
for previous applications of such tools
to invariant subgroups
for analytic automorphisms and endomorphisms.
We recall the concepts of local
stable manifolds, local unstable manifolds
and related notions in Section~\ref{secprel}.
\begin{thm}\label{thmA}
Let $\alpha\colon G\to G$ be a $\K$-analytic
endomorphism of a Lie group~$G$ over
a totally disconnected local field~$\K$.
Then the following holds:
\begin{itemize}
\item[\rm(a)]
There exists a unique $\K$-analytic manifold
structure on $\con(\alpha)$
making it a $\K$-analytic Lie group
$\con_*(\alpha)$, such that $\con_*(\alpha)$
has an open subset which is a local
stable manifold for~$\alpha$ around~$e$.
\item[\rm(b)]
There exists a unique $\K$-analytic manifold
structure on $\con^-(\alpha)$
making it a $\K$-analytic Lie group
$\con^-_*(\alpha)$, such that $\con^-_*(\alpha)$
has an open subset which is a local
unstable manifold for~$\alpha$ around~$e$.
\item[\rm(c)]
There exists a unique $\K$-analytic manifold
structure on $\lev(\alpha)$
making it a $\K$-analytic Lie group
$\lev_*(\alpha)$, such that $\lev_*(\alpha)$
has an open subset which is a
centre manifold for~$\alpha$ around~$e$.
\item[\rm(d)]
There exists a unique $\K$-analytic manifold
structure on $\parb(\alpha)$
making it a $\K$-analytic Lie group
$\parb_*(\alpha)$, such that $\parb_*(\alpha)$
has an open subset which is a
centre-stable manifold for~$\alpha$ around~$e$.
\item[\rm(e)]
There exists a unique $\K$-analytic manifold
structure on $\parb^-(\alpha)$
making it a $\K$-analytic Lie group
$\parb^-_*(\alpha)$, such that $\parb_*^-(\alpha)$
has an open subset which is a local
centre-unstable manifold for~$\alpha$ around~$e$.
\end{itemize}
\end{thm}
In the following three theorems,
we retain the situation of Theorem~\ref{thmA}.
\begin{thm}\label{thmB}
The manifolds just constructed have the following
properties.
\begin{itemize}
\item[\rm(a)]
Each of $\con_*(\alpha)$, $\con^-_*(\alpha)$,
$\lev_*(\alpha)$, $\parb_*(\alpha)$, and $\parb^-_*(\alpha)$
is an immersed Lie subgroup of~$G$
and
$\alpha$ restricts to a $\K$-analytic
endomorphism $\alpha_s$, $\alpha_u$, $\alpha_c$, $\alpha_{cs}$, and
$\alpha_{cu}$ thereon,
respectively.
\item[\rm(b)] We have
\begin{equation}\label{again-contr}
\con(\alpha_s)=\con_*(\alpha)\quad\mbox{and}\quad \con^-(\alpha_u)=\con^-_*(\alpha).
\end{equation}
Moreover, $\lev(\alpha_c)$, $\parb(\alpha_{cs})$,
and $\parb^-(\alpha_{cu})$
are open subgroups of $\lev_*(\alpha)$,
$\parb_*(\alpha)$, and $\parb^-_*(\alpha)$,
respectively.
\end{itemize}
\end{thm}
It is well known that $\parb(\alpha)$
normalizes $\con(\alpha)$
and $\parb^-(\alpha)$ normalizes
$\con^-(\alpha)$
(see \cite[Lemma~13.1\,(a)]{BGT}; cf.\ \cite[Proposition~3.4]{BaW}
for automorphisms).
\begin{thm}\label{thmC}
\begin{itemize}
\item[\rm(a)]
Let $\Omega\sub G$ be the big cell for~$\alpha$.
The product map
\[
\pi\colon \con_*(\alpha)\times\lev_*(\alpha)\times \con^-_*(\alpha)\to\Omega,
\quad (x,y,z)\mto xyz
\]
is an \'{e}tale $\K$-analytic map and surjective.
\item[\rm(b)]
The action $\lev_*(\alpha)\times \con_*(\alpha)\to\con_*(\alpha)$,
$(x,y)\mto xyx^{-1}$ is $\K$-analytic, whence
the product manifold structure turns
$\con_*(\alpha)\rtimes \lev_*(\alpha)$
into a $\K$-analytic Lie group.
The product map
\[
\con_*(\alpha)\rtimes \lev_*(\alpha)\to\parb_*(\alpha),\quad (x,y)\mto xy
\]
is a surjective group homomorphism and
an \'{e}tale $\K$-analytic map.
\item[\rm(c)]
The action $\lev_*^-(\alpha)\times \con_*^-(\alpha)\to\con_*^-(\alpha)$,
$(x,y)\mto xyx^{-1}$ is $\K$-analytic, whence
the product manifold structure turns
$\con_*^-(\alpha)\rtimes \lev_*(\alpha)$
into a $\K$-analytic Lie group.
The product map
\[
\con_*^-(\alpha)\rtimes \lev_*(\alpha)\to\parb_*^-(\alpha),\quad (x,y)\mto xy
\]
is a surjective group homomorphism and
an \'{e}tale $\K$-analytic map.
\end{itemize}
\end{thm}
If $H$ is a group and $\beta\colon H\to H$
an endomorphism, we write
\[
\ik(\beta):=\bigcup_{n\in\N_0}\ker(\beta^n)
\]
for the \emph{iterated kernel}.
If $G$ is a totally disconnected, locally
compact group and $\beta$ a continuous endomorphism,
then $\ik(\beta)\sub \con(\beta)$.
Let us call a continuous
endomorphism of a topological group~$G$
\emph{contractive} if $\alpha^n(g)\to e$
as $n\to\infty$,
for each $g\in G$.
If $G$ is a $\K$-analytic endomorphism
of a $\K$-analytic Lie group~$G$
and $\alpha$ is contractive, then $G=\con(\alpha)=\con_*(\alpha)$
(cf.\ \cite[Proposition~7.10\,(a)]{END}).
\begin{thm}\label{thmD}
\begin{itemize}
\item[\rm(a)]
If $\alpha$ is \'{e}tale,
then $\ik(\alpha)$ is discrete
in $\con_*(\alpha)$.
If $\K$ has characteristic~$0$,
then both properties are equivalent.
\item[\rm(b)]
If $\alpha$ is \'{e}tale, then
$\con_*(\alpha)/\ik(\alpha)$
is an open
$\beta$-invariant subgroup
for some $\K$-analytic
Lie group $H$ and contractive $\K$-analytic
automorphism $\beta\colon H\to H$
of~$H$ which extends the $\K$-analytic
endomorphism $\wb{\alpha}_s$ induced by
$\alpha_s$ on $\con_*(\alpha)/\ik(\alpha_s)$.
In particular,
$\con(\alpha)/\ik(\alpha)$
is nilpotent.
Moreover, $\con_*(\alpha)$ has an open,
$\alpha$-invariant subgroup~$U$ which is nilpotent.
\item[\rm(c)]
If $\car(\K)=0$,
then $\ik(\alpha_s)$ is
a Lie subgroup of $\con_*(\alpha)$.
Thus $Q:=\con_*(\alpha)/\ik(\alpha)$
has a unique $\K$-analytic manifold structure
making the canonical map $\con_*(\alpha)\to Q$
a submersion, and the latter turns~$Q$
into a $\K$-analytic Lie group.
There exists a $\K$-analytic
Lie group $H$ containing $Q$ as an open submanifold
and subgroup, and a contractive $\K$-analytic
automorphism $\beta\colon H\to H$
which extends the $\K$-analytic
endomorphism $\wb{\alpha}_s$ induced by
$\alpha_s$ on $Q$.
Notably, $\con(\alpha)/\ik(\alpha)$ is nilpotent.
\item[\rm(d)]
The $\K$-analytic surjective endomorphism
$\alpha_u$ of $\con^-_*(\alpha)$
is \'{e}tale.
Moreover, $\con^-_*(\alpha)$
has an open subgroup~$S$ with $S\sub\alpha_u(S)$
such that
$\alpha_u|_S\colon S\to\alpha_u(S)$
is a $\K$-analytic diffeomorphism
and $\alpha_u(S)$ can be regarded as an open submanifold
and subgroup of a $\K$-analytic Lie group~$H$
which admits a $\K$-analytic contractive
automorphism $\beta\colon H\to H$
such that $\alpha_u|_S=\beta^{-1}|_S$.
In particular, $S$ and $\alpha_u(S)$ are nilpotent.
\end{itemize}
\end{thm}
\begin{rem}
(a) If $\alpha\colon G\to G$ is a
$\K$-analytic automorphism,
it was known previously that $\con(\alpha)$
can be given an immersed
Lie subgroup structure modelled
on $\con(L(\alpha))$
such that $\alpha|_{\con(\alpha)}$
becomes a $\K$-analytic contractive automorphism;
likewise for $\con^-(\alpha)=\con(\alpha^{-1})$
(see \cite[Proposition~6.3\,(b)]{FIN}).
Hence $\con(\alpha)$ and $\con^-(\alpha)$
are nilpotent in this case (see \cite{CON}).\\[2mm]
(b) If $\alpha\colon G\to G$ is a $\K$-analytic
endomorphism and $\con(\alpha)$ is closed in~$G$,
then $\alpha|_{\con^-(\alpha)}$
is a $\K$-analytic automorphism of the Lie subgroup
$\con^-(\alpha)$ of~$G$
such that $(\alpha|_{\con^-(\alpha)})^{-1}$
is contractive (cf.\ \cite[Theorem~8.15]{END}), whence $\con^-(\alpha)$
is nilpotent. Moreover, $\alpha|_{\lev(\alpha)}$
is a $\K$-analytic automorphism of the
Lie subgroup $\lev(\alpha)$ of~$G$
and $\alpha|_{\lev(\alpha)}$ is distal
(see \cite[Theorem~8.15 and Proposition~7.10\,(c)]{END}).
\end{rem}
We also observe:
\begin{prop}\label{thmE}
In the setting of Theorem~\emph{\ref{thmA}},
the following holds:
\begin{itemize}
\item[\rm(a)]
$\ik(\alpha)\cap \lev_*(\alpha)$ is discrete.
\item[\rm(b)]
$\ker(\alpha)\cap \parb^-_*(\alpha)$
is discrete.
\item[\rm(c)]
$\ker(\alpha)\cap \con^-_*(\alpha)$
is discrete.
\end{itemize}
\end{prop}
We mention that refined topologies
on contraction subgroups of well-behaved
automorphisms (e.g., expansive automorphisms)
were also studied in \cite{Si2}
and~\cite{GaR}.\\[2.3mm]
The following examples illustrate
the results.
\begin{example}\label{any-gp}
Let $H$ be any group and
${\bf 1}\colon H\to H$, $x\mto e$
be the constant endomorphism.
If we endow $H$ with the discrete
topology, then $H$ becomes a $\K$-analytic
Lie group modeled on $\K^0$,
and ${\bf 1}$ a $\K$-analytic endomorphism such that
$\con({\bf 1})=H$.
As a consequence, contraction groups
of endomorphisms do not have any special
group-theoretic properties: any group can occur.
\end{example}
\begin{example}\label{Zp}
For a prime number~$p$, consider the group $(\Z_p,+)$
of $p$-adic integers, which is an open subgroup
of the local field $\Q_p$ of $p$-adic numbers
and thus a $1$-dimensional analytic Lie group over~$\Q_p$.
The map
\[
\beta\colon \Z_p\to\Z_p, \quad z\mto pz
\]
is a contractive $\Q_p$-analytic endomorphism.
Since $\beta$ is injective, $\ik(\beta)=\{0\}$
is trivial and hence discrete.
We can extend $\beta$ to the analytic contractive automorphism
$\gamma\colon \Q_p\to\Q_p$, $z\mto pz$.
\end{example}
\begin{example}
If $H$ in Example~\ref{any-gp} is non-trivial,
then
\[
\alpha:={\bf 1}\times \beta
\colon H\times \Z_p\to H\times \Z_p
\]
is a contractive $\Q_p$-analytic endomorphism
which cannot extend to a contractive
automorphism of a Lie group $G\supseteq H\times\Z_p$
as $\alpha$ is not injective.
However, $\{e\}\times \Z_p\cong \Z_p$
is an open $\alpha$-invariant subgroup and embeds in
$(\Q_p,\gamma)$.
Likewise, $G/\ik(\alpha)=G/(H\times\{0\})\cong \Z_p$
embeds in $(\Q_p,\gamma)$.
\end{example}
\begin{example}\label{two-sided}
Let $\F$ be a finite field,
$\K:=\F(\!(X)\!)$ be the local field
of formal Laurent series over~$\F$
and
$\F[\![X]\!]$ be its compact open subring
of formal power series.
Then $G:=\F^\Z$, with the compact product topology,
can be made a $2$-dimensional
$\K$-analytic Lie group in such a way that the bijection
\[
\phi \colon G\to X \F[\![X]\!]\times \F[\![X]\!],\quad
(a_k)_{k\in\Z}\mto
\left(
\sum_{k=1}^\infty a_{-k} X^k,
\sum_{k=0}^\infty a_k X^k \right)
\]
becomes a $\K$-analytic diffeomorphism.
The right-shift $\alpha\colon G\to G$, $(a_k)_{k\in\Z}\mto (a_{k-1})_{k\in\Z}$
is an endomorphism and $\K$-analytic,
as $\phi\circ \alpha\circ\phi^{-1}$
is the restriction to an open set of the
$\K$-linear map
\[
\K\times \K\to\K\times \K,\quad
(z,w)\mto (X^{-1}z,Xw).
\]
Here $\con(\alpha)$ is the dense proper subgroup
of all $(a_k)_{k\in\N}$
of support bounded to the left, i.e.,
there is $k_0\in\Z$ such that $a_k=0$ for all
$k\leq k_0$. Moreover,
\[
\con_*(\alpha)=\F(\!(X)\!).
\]
Likewise, $\con^-(\alpha)$
is the set of all sequences with support bounded
to the right and $\con^-_*(\alpha)\cong\F(\!(X)\!)$
with the automorphism $z\mto X^{-1}z$.
Note that
$\con(\alpha)\cap\con^-(\alpha)$
is the subgroup of all finitely-supported sequences,
which is dense in~$G$ and dense in both
$\con_*(\alpha)$ and $\con_*^-(\alpha)$.
As $G$ is compact, we have $\lev(\alpha)=G$.
Moreover, $\lev^*(\alpha)=G$, endowed with
the discrete topology.
\end{example}
\begin{example}
For $\K$ as in Example~\ref{two-sided},
$G:=\F[\![X]\!]$ is an open subgroup of $(\K,+)$.
The left-shift
\[
\alpha\colon G\to G, \quad \sum_{k=0}^\infty a_kX^k\mto \sum_{k=0}^\infty a_{k+1}X^k
\]
is an endomorphism of~$G$ and $\K$-analytic as it coincides
with the $\K$-linear map $\K\to\K$, $z\mto X^{-1}z$
on the open subgroup $X\F[\![X]\!]$.
We have $\con^-(\alpha)=\F[\![X]\!]$,
since $(X^nz)_{n\in\N_0}$ is an $\alpha$-regressive
trajectory for $z\in G$ with $X^nz\to 0$ as $n\to\infty$.
Moreover,
$\con(\alpha)\sub G$ is the dense subgroup
of finitely supported sequences,
$\ker(\alpha)=\F X^0$, $\ik(\alpha)=\con(\alpha)$,
and $\lev(\alpha)=G$.
Also note that $\con^-(\alpha)=\con^-_*(\alpha)$,
while $\con_*(\alpha)$ and $\lev_*(\alpha)$
are discrete. 
\end{example}
\begin{example}
If $\K$ is a totally disconnected local field
and $\alpha$ a $\K$-linear endomorphism
of a finite-dimensional $\K$-vector space~$E$,
then $E$ admits the Fitting decomposition
\[
E=\ik(\alpha)\oplus F
\]
with $F:=\bigcap_{k\in\N_0}\alpha^k(E)$.
\end{example}
For a non-discrete $\K$-analytic Lie group~$G$
and a $\K$-analytic endomorphism $\alpha\colon G\to G$
it can happen
that both $\bigcap_{k\in\N_0}\alpha^k(G)=\{e\}$
and $\ik(\alpha)=\{e\}$,
even in the case $\K=\Q_p$
(see Example~\ref{Zp}).\footnote{We mention
that the corresponding
Lie algebra endomorphism $L(\alpha)\colon L(G)\to L(G)$, $z\mto pz$
is an isomorphism; thus $\ik(L(\alpha))=\{0\}$ and
$\bigcap_{k\in\N_0}L(\alpha)^k(L(G))=L(G)$.}
In the following example, $\ik(\alpha)=\{e\}$ holds,
$\bigcap_{k\in\N_0}\alpha^k(G)=\{e\}$ and $L(\alpha)=0$,
whence $\ik(L(\alpha))=L(G)$ and $\bigcap_{k\in\N_0}L(\alpha)^k(L(G))=\{0\}$.
\begin{example}
If $\F$ is a finite field of positive characteristic~$p$,
then the Frobenius homomorphism
$\F(\!(X)\!)\to\F(\!(X)\!)$, $z\mto z^p$
is a $\K$-analytic endomorphism of $(\F(\!(X)\!),+)$
and restricts to an injective contractive endomorphism~$\alpha$
of the compact open subgroup $G:=X\F[\![X]\!]$ of
$\F(\!(X)\!)$.
Since $\frac{d}{dz}\big|_{z=0}(z^p)=pz^{p-1}|_{z=0}=0$,
we have $L(\alpha)=0$.
\end{example}
As shown in \cite{CON},
a $\K$-analytic Lie group~$G$ is nilpotent
if it admits
a contractive $\K$-analytic automorphism~$\alpha$.
This becomes false for endomorphisms;
not even open subgroups
need to exist in this case which are nilpotent.
\begin{example}
Let $\F_2$ be the field with $2$ elements.
Then $S:=\SL_2(\F(\!(X)\!)$
is a totally disconnected,
locally compact group which is compactly generated
and simple as an abstract group.
Hence none of the open subgroups
of~$S$ is soluble (see \cite{WiS}),
whence none of them is nilpotent.
Let $G\sub S$ be the congruence subgroup
consisting of all $(a_{ij})_{i,j=1}^2\in S$
such that $a_{ij}-\delta_{ij}\in X\F[\![X]\!]$
for all $i,j\in\{1,2\}$.
Then $G$ is an open subgroup of~$S$
and the map
\[
\alpha\colon G\to G,\;\;
\left(
\begin{array}{cc}
a & b\\
c & d
\end{array}
\right)\mto
\left(
\begin{array}{cc}
a^2 & b^2\\
c^2 & d^2
\end{array}
\right),
\]
which applies the Frobenius to each matrix element,
is a contractive $\K$-analytic endomorphism
and injective.
No open subgroup of~$G$ is nilpotent.
\end{example}
The pathology also occurs if $\car(\K)=0$.
\begin{example}
Let $G:=\PSL_2(\Q_p)$
and $\alpha\colon G\to G$
be the trivial endomorphism $g\mto e$.
Then $G=\con(\alpha)=\con_*(\alpha)$.
As $G$ is a totally disconnected group
which is compactly generated and simple,
no open subgroup of~$G$ is soluble
(nor nilpotent, as a special case).
\end{example}
\section{Preliminaries and notation}\label{secprel}
We write $\N:=\{1,2,\ldots\}$ and $\N_0:=\N\cup\{0\}$.
All topological groups considered in this work
are assumed Hausdorff.
If $X$ is a set, $Y\sub X$ a subset
and $f\colon Y\to X$ a map,
we call a subset $M\sub Y$
\emph{invariant} under $f$ if $f(M)\sub M$;
we say that $M$ is \emph{$f$-stable} if $f(M)=M$.
In the following, $1/0:=\infty$.
Recall that a topological field~$\K$ is called
a \emph{local field} if it is non-discrete
and locally compact~\cite{Wei}.
If~$\K$ is a totally disconnected local field,
we fix an absolute value $|\cdot|$ on~$\K$
which defines its topology
and extend it to an absolute value, also denoted
$|\cdot|$, on an algebraic closure~$\wb{\K}$ of~$\K$
(see \cite[\S14]{Sch}).
If $E$ is a vector space over~$\K$, endowed with an ultrametric norm
$\|\cdot\|$, we write
\[
B_r^E(x):=\{y\in E\colon \|y-x\|<r\}
\]
for the ball of radius $r>0$ around $x\in E$.
If $\alpha\colon E\to F$
is a continuous linear map between
finite-dimensional $\K$-vector spaces~$E$ and $F$,
endowed with ultrametric norms $\|\cdot\|_E$ and $\|\cdot\|_F$,
respectively, we let
\[
\|\alpha\|_{\op}:=\sup\left\{\frac{\|\alpha(x)\|_F}{\|x\|_E}\colon x\in E\setminus \{0\}\right\}\in [0,\infty[
\]
be its operator norm.
For the basic theory of $\K$-analytic mappings
between open subsets of finite-dimensional
$\K$-vector spaces and the corresponding
$\K$-analytic manifolds and Lie groups modeled on $\K^m$,
we refer to \cite{Ser}
(cf.\ also \cite{FAS} and \cite{Bou}).
We shall write $T_pM$ for the tangent
space of a $\K$-analytic manifold~$M$
at $p\in M$.
Given a $\K$-analytic map
$f\colon M\to N$ between $\K$-analytic manifolds,
we write
$T_pf\colon T_pM\to T_{f(p)}M$
for the tangent map of $f$ at $p\in M$.
Submanifolds are as in \cite[Part~II, Chapter~III,
\S11]{Ser}.
If $N$ is a submanifold of a $\K$-analytic manifold~$M$
and $p\in N$, then the inclusion map $\iota\colon N \to M$
is an immersion and we identify $T_pN$ with the vector subspace $T_p\iota(T_p(N))$
of~$T_pM$.
If $F\sub T_pM$ is a vector subspace and
$T_pN=F$, we say that \emph{$N$ is tangent to~$F$} at $p$.
If $M$ is a set, $p\in M$
and $N_1$, $N_2$ are $\K$-analytic
manifolds such that $N_1,N_2\sub M$
as a set and $p\in N_1\cap N_2$,
we write $N_1\sim_p N_2$
if there exists a subset $U\sub N_1\cap N_2$
which is an open $p$-neighbourhood in both $N_1$
and $N_2$, and such that $N_1$ and $N_2$ induce the same
$\K$-analytic manifold structure on~$U$. Then $\sim_p$
is an equivalence relation;
the equivalence class of $N_1$ with respect to
$\sim_p$ is called the \emph{germ of $N_1$ at~$p$}.
If $G$ is a $\K$-analytic Lie group,
we let $L(G):=T_eG$, endowed with its natural
Lie algebra structure;
if $f\colon G\to H$
is a $\K$-analytic homomorphism between
Lie groups, we abbreviate $L(f):=T_e(f)\colon L(G)\to L(H)$.
If $G$ is a $\K$-analytic Lie group
and a subgroup $H\sub G$ is endowed
with a $\K$-analytic Lie group structure
turning the inclusion map
$H\to G$ into an immersion,
we call $H$ an \emph{immersed Lie subgroup}
of~$G$.
This holds if and only if $H$
has an open subgroup
which is a submanifold of~$G$.
If we call a map $f\colon M\to N$
between $\K$-analytic manifolds
a $\K$-analytic diffeomorphism,
then $f^{-1}$ is assumed $\K$-analytic as well.
Likewise, $\K$-analytic isomorphisms between
$\K$-analytic Lie groups (notably, $\K$-analytic
automorphisms) presume a $\K$-analytic inverse function.
If $U$ is an open subset of
a finite-dimensional $\K$-vector space, we identify
the tangent bundle $U$ with $U\times E$,
as usual.
If $f\colon M\to U$
is a $\K$-analytic map, we write
$df\colon TM\to E$ for the second component of
the tangent map $Tf\colon TM\to TU=U\times E$.
If $X$ and $Y$ are topological spaces and $x\in X$,
we say that a map $f\colon X\to Y$ is
\emph{open at $x$} if $f(U)$ is an $f(y)$-neighbourhood in~$Y$
for each $x$-neighbourhood $U$ in~$X$.
If $G$ is a group with group multiplication $(x,y)\mto xy$, we write
$G^{\op}$ for $G$ endowed with the opposite
group structure, with multiplication $(x,y)\mto yx$.\\[1mm]
We shall use several basic facts
(with proofs recalled in Appendix~\ref{appA}).
The terminology in~(d) is as in \cite[Part~II, Chapter~III, \S9]{Ser}.
\begin{numba}\label{fact-2}
Let $\sigma\colon G\times X\to X$, $(g,x)\mto g.x$ be a continuous left action
of a topological group~$G$ on a topological space~$X$,
and $x\in X$. Then we have:
\begin{itemize}
\item[(a)]
If the orbit $G.x$ has an interior point,
then $G.x$ is open in~$X$.
\item[(b)]
If $\sigma^x\colon G\to G.x$, $g\mto g.x$
is open
at~$e$, then $\sigma^x$ is an open map.
\item[(c)]
If $G$ is compact and~$X$ is Hausdorff,
then $\sigma^x\colon G\to G.x$ is an open map.
\item[(d)]
If $G$ is a $\K$-analytic Lie group, $X$ a $\K$-analytic manifold,
$\sigma$ is $\K$-analytic, $G.x\sub X$ is open and $\sigma^x$
is \'{e}tale at $e$, then $\sigma^x$ is \'{e}tale.
\end{itemize}
\end{numba}
We shall also use the following fact concerning automorphic actions.
\begin{numba}\label{act-ana}
Let $G$ and $H$ be $\K$-analytic Lie groups
and $\sigma\colon G\times H\to H$ be a left $G$-action on~$H$
with the following properties:
\begin{itemize}
\item[(a)]
$\sigma_g:=\sigma(g,\cdot)\colon H\to H$ is an automorphism
of the group~$H$, for each $g\in G$.
\item[(b)]
For each $g\in G$, there exists an open $e$-neighbourhood
$Q\sub H$ such that $\sigma_g|_Q$ is $\K$-analytic.
\item[(c)]
For each $x\in H$, there exists
an open $e$-neighbourhood $P\sub G$ such that $\sigma^x:=\sigma(\cdot,x)\colon G\to H$
is $\K$-analytic on~$P$.
\item[(d)]
There exists an open $e$-neighbourhood $U\sub G$
and an open $e$-neighbourhood $V\sub H$ such that $\sigma|_{U\times V}$
is $\K$-analytic.
\end{itemize}
Then $\sigma$ is $\K$-analytic.
\end{numba}
Again, a proof can be found in Appendix~\ref{appA}.
Likewise for the following fact.
\begin{numba}\label{comp-to-id}
Let $G$ be a topological group,
$(g_n)_{n\in\N}$ be a sequence in~$G$
such that $\{g_n\colon n\in \N\}$
is relatively compact
and $(x_n)_{n\in\N}$ be a sequence in~$G$
such that $x_n\to e$ as $n\to\infty$.
Then $g_nx_ng_n^{-1}\to e$.
\end{numba}
\begin{numba}\label{thesubsp}
Henceforth, let $\K$ be a totally disconnected local field
with algebraic closure~$\wb{\K}$ and absolute value~$|\cdot|$.
If~$E$ is a finite-dimensional $\K$-vector space
and $\alpha\colon E\to E$ a $\K$-linear endomorphism,
call $\rho\in [0,\infty]$
a \emph{characteristic value} of~$\alpha$
if $\rho=|\lambda|$ for some
eigenvalue $\lambda\in\wb{\K}$ of the endomorphism
$\alpha_{\wb{\K}}:=\alpha\otimes_\K \id_{\wb{\K}}$ of $E\otimes_\K\wb{\K}$.
If $R(\alpha)\sub[0,\infty[$ is the set of all characteristic
values of~$\alpha$, then
\[
E=\bigoplus_{\rho \in R(\alpha)}E_\rho
\]
for unique $\alpha$-invariant vector subspaces $E_\rho\sub E$
such that $E_\rho\otimes_\K\wb{K}$
equals the sum of all generalized
eigenspaces of $\alpha_{\wb{K}}$ for eigenvalues
$\lambda\in\wb{K}$ with $|\lambda|=\rho$
(compare \cite[Chapter~II, (1.0)]{Mar}).
If $a\in]0,\infty[$ such that $a\not\in R(\alpha)$,
we say that $\alpha$ is \emph{$a$-hyperbolic}.
For $\rho\in [0,\infty[\setminus R(\alpha)$,
let $E_\rho:=\{0\}$.
For $\rho\in [0,\infty[$,
we call $E_\rho$ the \emph{characteristic subspace}
of~$E$ for $\rho$.
Then $E_0=\ik(\alpha)$. Moreover,
$\alpha(E_\rho)=E_\rho$
for each $\rho>0$ and $\alpha|_{E_\rho}\colon E_\rho\to E_\rho$
is an isomorphism.
For each $a\in \,]0,\infty[$, we consider
the following $\alpha$-invariant vector subspaces of~$E$:
\[
E_{<a}:=\bigoplus_{\rho<a}E_\rho,\quad
E_{\leq a}:=\bigoplus_{\rho\leq a}E_\rho,\quad
E_{>a}:=\bigoplus_{\rho>a}E_\rho\quad\mbox{and}\quad
E_{\geq a}:=\bigoplus_{\rho\geq a}E_\rho.
\]
%
%
%
\end{numba}
\begin{numba}\label{def-adap}
By \cite[Proposition~2.4]{FIN},
$E$ admits an ultrametric norm $\|\cdot\|$
which is \emph{adapted to $\alpha$}
in the following sense:
\begin{itemize}
\item[(a)]
$\|x\|=\max\{\|x_\rho\|\colon \rho\in R(\alpha)\}$
if we write $x\in E$ as $x=\sum_{\rho\in R(\alpha)}x_\rho$
with $x_\rho\in E_\rho$;
\item[(b)]
$\|\alpha|_{E_0}\|_{\op}<1$;
\item[(c)]
For all $\rho\in R(\alpha)$ such that $\rho>0$,
we have $\|\alpha(x)\|=\rho\|x\|$ for all $x\in E_\rho$.
\end{itemize}
If $\ve\in\,]0,1]$ is given, then an adapted norm
can be found such that, moreover, $\|\alpha|_{E_0}\|_{\op}<\ve$.
\end{numba}
\begin{rem}\label{forifthm}
By~(a) in~\ref{def-adap},
we have
\[
B^E_r(0)=\prod_{\rho\in R(\alpha)}
B^{E_\rho}_r(0)
\]
for each $r>0$, identifying $E$ with $\prod_{\rho\in R(\alpha)}E_\rho$.
For each $\rho\in R(\alpha)\setminus \{0\}$,
(c) implies that
\[
\alpha(B^{E_\rho}_r)=B_{\rho r}^{E_\rho}(0)\quad\mbox{for all $\, \rho\in R(\alpha)\setminus\{0\}$.}
\]
\end{rem}
\begin{numba}\label{thesetinvma}
Let $M$ be a $\K$-analytic manifold
and $p\in M$. Let $M_0\sub M$ be an open
$p$-neighbourhood and $f\colon M_0\to M$
be a $\K$-analytic map such that $f(p)=p$.
Let $T_p(M)_\rho$ for $\rho>0$
be the characteristic subspaces of $T_pM$
with respect to the endomorphism $T_pf$ of $T_pM$.
\end{numba}
\begin{numba}
Let $N\sub M_0$ be a submanifold such that $p\in N$.
\begin{itemize}
\item[(a)]
If $a\in\,]0,1]$
and $T_pf$ is $a$-hyperbolic,
we say that $N$ is a \emph{local $a$-stable manifold}
for $f$ around~$p$ if $T_pN=(T_pM)_{<a}$ and
$f(N)\sub N$.
If, moreover, $a>\rho$ for each
$\rho\in R(T_pf)$ such that $\rho<1$,
we call $N$ a \emph{local stable manifold}
for $f$ around~$p$.
\item[(b)]
We say that $N$ is a \emph{centre manifold}
for $f$ around~$p$ if $T_pN=(T_pM)_1$
and $f(N)=N$.
\item[(c)]
If $b\geq 1$ and $T_pf$ is $b$-hyperbolic,
we say that $N$ is a \emph{local $b$-unstable manifold} for~$f$
around~$p$ if $T_pN=(T_pM)_{>b}$
and there exists an open $p$-neighbourhood $P\sub N$
such that $f(P)\sub N$.
If, moreover, $b<\rho$ for each
$\rho\in R(T_pf)$ such that $\rho>1$,
we call $N$ a \emph{local unstable manifold}
for $f$ around~$p$.
\item[(d)]
We call $N$
a \emph{centre-stable manifold}
for $f$ around~$p$ if $f(N)\sub N$
and $T_pN=(T_pM)_{\leq 1}$.
\item[(e)]
We call $N$ a \emph{local centre-unstable manifold}
for $f$ around~$p$ if $T_pN=(T_pM)_{\geq 1}$
and there
exists an open $p$-neighbourhood
$P\sub N$ such that $f(P)\sub N$.
\end{itemize}
\end{numba}
\begin{numba}\label{uni-germ}
We mention that a centre manifold
for $f$ around~$p$ always exists in the situation
of~\ref{thesetinvma},
by \cite[Theorem~1.10]{INV},
whose germ at~$p$ is uniquely determined
(noting that the alternative
argument in the proof
does not require that $T_pf$ be an automorphism).
For each $a\in\,]0,1]$ such that $T_pf$ is
$a$-hyperbolic,
a local $a$-stable manifold
around $p$ exists,
whose germ around~$p$ is uniquely
determined (by the Local
Invariant Manifold Theorem in~\cite{FIN}).
For each $b\in [1,\infty[$
such that $T_pf$ is
$b$-hyperbolic,
a local $b$-unstable manifold
around $p$ exists,
whose germ around~$p$ is uniquely
determined (by the Local Invariant Manifold
Theorem just cited).
Moreover,
a centre-stable manifold
for $f$ around~$p$
always exists, whose germ
at~$p$ is uniquely determined
(again by the Local Invariant Manifold
Theorem).
\end{numba}
\begin{numba}\label{also-nonhypo}
The germ at $p$ of a local stable
manifold for $f$ around~$p$
is uniquely determined.
\,In fact,
\[
(T_pM)_{<a}=(T_pM)_{<b}
\]
for all $a,b\in\,]0,1]\setminus R(T_p(f))$ with $a>\rho$ and $b>\rho$
for all $\rho\in R(T_pf)\cap \,[0,1[$.
As a consequence, a submanifold
$N\sub M$ is a local $a$-stable manifold
if and only if it is a local $b$-stable
manifold for~$f$ around~$p$. Thus \ref{uni-germ}
applies.
\end{numba}
Likewise, $(T_pM)_{>a}=(T_pM)_{>b}$
for all $a,b\in [1,\infty[\,\setminus R(T_pf)$ such that $a<\rho$ and $b<\rho$
for all $\rho\in R(T_pf)\cap \,]1,\infty[$.
Hence, a submanifold $N\sub M$
is a local $a$-unstable manifold
for $f$ around~$p$
if and only if it is a local $b$-unstable
manifold. As a consequence:
\begin{numba}\label{nonhypo2}
The germ at~$p$ of a local unstable
manifold for $f$ around~$p$
is uniquely determined.
\end{numba}
Let $M$ be a $\K$-analytic manifold,
$f\colon M\to M$ be $\K$-analytic
and $p\in M$ such that $f(p)=p$.
Let $a\in\,]0,1]\setminus R(T_pf)$,
$\,b\in [1,\infty[\setminus R(T_pf)$
and $\|\cdot\|$
be an ultrametric norm on $E:=T_pM$
adapted to $T_pf$
such that $\|T_pf|_{E_0}\|_{\op}<a$.
Endow vector subspaces $F\sub E$ with the norm
induced by $\|\cdot\|$
and abbreviate $B^F_t:=B^F_t(0)$ for $t>0$.
We shall use the following fact, which is \cite[Proposition~7.3]{END}:
\begin{numba}\label{prop-7-3}
There exists $R>0$ with the following properties:
\begin{itemize}
\item[(a)]
There exists a local $a$-stable manifold $W^s_a$
for~$f$ around~$p$ and a $\K$-analytic diffeomorphism
\[
\phi_s\colon W_a^s\to B_R^{E_{<a}}
\]
such that $\phi_s(p)=0$, $W^s_a(t):=\phi_s^{-1}(B_t^{E_{<a}})$
is a local $a$-stable manifold for~$f$ around $p$
for all $t\in \,]0,R]$,
and $d\phi_s|_{T_p(W^s_a)}=\id_{E_{<a}}$.
\item[(b)]
There exists a centre manifold $W^c$
for~$f$ around~$p$ and a $\K$-analytic diffeomorphism
\[
\phi_c\colon W^c\to B_R^{E_1}
\]
such that $\phi_c(p)=0$, $W^c(t):=\phi_c^{-1}(B_t^{E_1})$
is a centre manifold for~$f$ around $p$
for all $t\in \,]0,R]$,
and $d\phi_c|_{T_p(W^c)}=\id_{E_1}$.
\item[(c)]
There exists a local $b$-unstable manifold $W^u_b$
for~$f$ around~$p$ and a $\K$-analytic diffeomorphism
\[
\phi_u\colon W_b^u\to B_R^{E_{>b}}
\]
such that $\phi_u(p)=0$, $W^u_b(t):=\phi_u^{-1}(B_t^{E_{>b}})$
is a local $b$-unstable manifold for~$f$ around $p$
for all $t\in \,]0,R]$,
and $d\phi_u|_{T_p(W^u_b)}=\id_{E_{>b}}$.
\end{itemize}
\end{numba}
See Appendix~\ref{appA} for a proof of the following
auxiliary result.
\begin{la}\label{enough-uni}
Let $M$ be a $\K$-analytic manifold,
$p\in M$ and $f\colon M_0\to M$
be a $\K$-analytic mapping
such that $f(p)=p$.
If $N\sub M$ is a
local centre-unstable manifold for~$f$
around~$p$, then for every $p$-neighbourhood
$W\sub N$, there exists an open $p$-neighbourhood
$O\sub N$ such that $f(O)$ is open in~$N$
and $O\sub f(O)\sub W$.
In particular, $f(O)$ is a local centre-unstable
manifold for $f$ around~$p$.
\end{la}
\section{Proof of Proposition~\ref{fi-res}}\label{fipro}
Let $U\sub G$ be a compact open subgroup
which minimizes the index
\[
[\alpha(U):\alpha(U)\cap U].
\]
Then $U$ is \emph{tidy} for $\alpha$ in the sense of \cite[Definition~2]{Wi3}
(by the main theorem of~\cite{Wi3}, on p.\,405).
Thus
\[
U=U_+U_-=U_-U_+,
\]
where $U_+$ is the subgroup of all $x\in U$ having an $\alpha$-regressive
trajectory in~$U$ and $U_-:=\{x\in U\colon (\forall n\in\N)\; \alpha^n(x)\in U\}$.
Moreover, $U_+$ and $U_-$
are compact subgroups of~$U$
(see \cite[Definition 4 and Proposition~1]{Wi3}).
Now
\[
\parb(\alpha)\cap U=U_-\quad\mbox{and}\quad
\parb^-(\alpha)\cap U=U_+,
\]
by \cite[Proposition~11]{Wi3},
whence $U_+$ and $U_-$ are compact open subgroups
of $\parb^-(\alpha)$ and $\parb(\alpha)$, respectively.
By \cite[Lemma~13.1]{BGT}, we have
\[
\parb(\alpha)=\con(\alpha)\lev(\alpha)\quad\mbox{and}\quad
\parb^-(\alpha)=\con^-(\alpha)\lev(\alpha)=\lev(\alpha)\con^-(\alpha),
\]
entailing that
\[
\Omega:=\con(\alpha)\lev(\alpha)\con^-(\alpha)=\parb(\alpha)\parb^-(\alpha).
\]
The product map
\[
p\colon \parb(\alpha)\times \parb^-(\alpha)\to G,\quad (x,y)\mto xy
\]
is continuous, with image~$\Omega$.
We get a continuous left action $\sigma\colon H\times G\to G$
of the direct product
\[
H:=\parb(\alpha)\times(\parb^-(\alpha)^{\op})
\]
on~$G$ via $(x,y).z:=xzy$.
Then $\Omega=H.e$ equals the $e$-orbit.
As $\Omega\supseteq U_-U_+=U$ is
a neighbourhood of~$e$ in~$G$,
\ref{fact-2}\,(a)
shows that $\Omega$ is open in~$G$.
Note that the orbit map
\[
\sigma^e\colon H\to G,\quad (x,y)\mto xey=xy
\]
equals~$p$.
We now restrict $\sigma$
to a continuous left action $\tau\colon K\times G\to G$ of the compact
group
\[
K:= U_-\times (U_+)^{\op}
\]
on $G$. Then $p(K)=U=K.e$ and
\[
p|_K\colon K\to K.e
\]
is the orbit map, which is an open map by \ref{fact-2}\,(c).
Since $p(K)=U$ is open in~$G$,
we deduce that also the map $p\colon H\to H.e=\Omega$
is open at~$e$ end hence an open map,
by \ref{fact-2}\,(b). $\,\square$
\section{Local structure of {\boldmath$(G,\alpha)$} around {\boldmath$e$}}\label{locstru}
We recall the construction of well-behaved $e$-neighbourhoods
from~\cite[\S8]{END}. These facts are essential for
all of our main results.
Let $G$ be a Lie group over a totally disconnected local field~$\K$
and $\alpha\colon G\to G$ be a $\K$-analytic endomorphism.
Pick an ultrametric norm $\|\cdot\|$ on
$\cg:=L(G)$ which is adapted to
$L(\alpha)$.
\begin{numba}
Pick $a\in\,]0,1]$
such that $L(\alpha)$ is $a$-hyperbolic
and $a>\rho$ for each characteristic value~$\rho$
of $L(\alpha)$ such that $\rho<1$.
Pick $b\in [1,\infty[$ such that
$L(\alpha)$ is $b$-hyperbolic and
$b<\rho$ for each charcteristic value $\rho$
of $L(\alpha)$ such that $\rho>1$.
With respect to the endomorphism $L(\alpha)$ of~$\cg$,
we then have
\[
\cg_{<1}=\cg_{<a}\quad\mbox{and}\quad \cg_{>1}=\cg_{>b},
\]
entailing that
\[
\cg=\cg_{<a}\oplus \cg_1\oplus \cg_{>b}.
\]
We find it useful to identify $\cg$ with the direct product
$\cg_{<a}\times\cg_1\times\cg_{>b}$;
an element $(x,y,z)$ of the latter is identified with
$x+y+z\in\cg$.
Let $R>0$, $W^s_a$, $W^c$, $W^u_b$
and the $\K$-analytic diffeomorphisms
\[
\phi_s\colon W^s_a\to B^{\cg_{<a}}_R(0),\quad
\phi_c\colon W^c\to B^{\cg_1}_R(0),\quad\mbox{and}\quad
\phi_u\colon W^u_b\to B^{\cg_{>b}}_R(0)
\]
be as in \ref{prop-7-3}, applied with $G$ in place of~$M$,
$\alpha$ in place of~$f$ and $e$ in place of~$p$.
Abbreviate $B^F_t:=B^F_t(0)$ if $t>0$ and $F\sub\cg$
is a vector subspace.
Using the inverse maps $\psi_s:=\phi_s^{-1}$, $\psi_c:=\phi_c^{-1}$,
and $\psi_u:=\phi_u^{-1}$, we define the $\K$-analytic map
\[
\psi\colon B_R^\cg=B_R^{\cg_{<a}}\times B^{\cg_1}_R\times
B^{\cg_{>b}}_R\to G,\quad (x,y,z)\mto \psi_s(x)\psi_c(y)\psi_u(z).
\]
Then $T_0\psi=\id_\cg$ if we identify $T_0\cg=\{0\}\times\cg$
with $\cg$ by forgetting the first component.
By the inverse function theorem,
after shrinking~$R$ if necessary, we may assume that
the image $W^s_aW^cW^u_b$ of~$\psi$
is an open identity neighbourhood in~$G$,
and that
\[
\psi\colon B^\cg_R\to W^s_aW^cW^u_b
\]
is a $\K$-analytic diffeomorphism.
In particular,
the product map
\begin{equation}\label{localprod}
W^s_a\times W^c\times W^u_b\to W^s_aW^cW^u_b,\quad
(x,y,z)\mto xyz
\end{equation}
is a $\K$-analytic diffeomorphism.
We define $\phi:=\psi^{-1}$,
with domain $U:=W^s_aW^cW^u_b$ and image $V:=B^\cg_R$.
After shrinking $R$ further if necessary,
we may assume that
\[
B^\phi_t:=\phi^{-1}(B^\cg_t)
\]
is a compact open subgroup of~$G$
for each $t\in\,]0,R]$
and a normal subgroup of~$B^\phi_R$
(see \cite[5.1 and Lemma~5.2]{END}).
Then
\[
B^\phi_t=\phi^{-1}(B^\cg_t)=\psi(B^\cg_t)
=\psi_s(B^{\cg_{<a}}_t)\psi_c(B^{\cg_1}_t)\psi_u(B^{\cg_{>b}}_t)
=W^s_a(t)W^c(t)W^u_b(t)
\]
with notation as in \ref{prop-7-3}.
\end{numba}
\begin{numba}\label{nunum}
After shrinking~$R$ if necessary,
we may assume that $W^u_b(t)$
is a subgroup of~$G$ for all $t\in\, ]0,R]$
and the set $W^c(t):=\phi_c^{-1}(B^{\cg_1}_t)$
normalizes $W^u_b(t)$
(see \cite[Lemma~8.7]{END}).
Since $W^s_a(t)$ is a local $a$-stable manifold for~$\alpha$,
we have
\[
\alpha(W^s_a(t))\sub W^s_a(t)\quad\mbox{for all $t\in\,]0,R]$.}
\]
After shrinking~$R$, moreover
\begin{equation}\label{subcontra}
\bigcap_{n\in\N_0}\alpha^n(W^s_a)=\{e\}\quad\mbox{and}\quad
\lim_{n\to\infty}\alpha^n(x)=e\;\,\mbox{for all $x\in W^s_a$}
\end{equation}
(see \cite[8.8]{END}). In addition, one may assume that
\[
\alpha|_{W^c(t)}\colon W^c(t)\to W^c(t)
\]
is a $\K$-analytic diffeomorphism for each $t\in\,]0,R]$
(see \cite[8.9]{END}).
\end{numba}
\begin{numba}\label{subu}
As shown in \cite[8.10]{END},
after shrinking~$R$ one may assume
that, for each $t\in\,]0,R]$
and $x\in W^u_b(t)$,
there exists an $\alpha$-regressive
trajectory $(x_{-n})_{n\in\N_0}$
in~$W^u_b(t)$ such that $x_0=x$ and
\[
\lim_{n\to\infty} x_{-n}=e;
\]
moreover, one may assume that $W^u_b(t)\sub\alpha(W^u_b(t))$ for all $t\in\,]0,R]$.
Since $L(\alpha|_{W^u_b})=L(\alpha)|_{\cg_{>b}}$
is injective,
we may assume that $\alpha|_{W^u_b}$
is an injective immersion,
after possibly shrinking~$R$.
\end{numba}
\begin{numba}
After shrinking~$R$ if necessary,
we may assume that $W^s_a(t)$ is a subgroup
of~$G$ for each $t\in\,]0,R]$
and that $W^c(t)$ normalizes $W^s_a(t)$.
Using a dynamical description
of the local $a$-stable manifolds
as in \cite[Theorem~6.6(c)(i)]{INV}, this can be proved like
\cite[Lemma~8.7]{END}.
\end{numba}
\begin{numba}
By \cite[8.1]{END},
there exists $r\in\,]0,R]$
such that $\alpha(W^u_b(r))\sub W^u_b$
and, for each $x\in W^u_b(r)\setminus\{e\}$,
there exists $n\in \N$
such that $\alpha^n(x)\not\in W^u_b(r)$.
\end{numba}
\begin{numba}
For each $t\in \,]0,R]$,
\[
(B_t^\phi)_-:=\{x\in B^\phi_t\colon (\forall n\in\N)\colon
\alpha^n(x)\in B^\phi_t\}
\]
is a compact subgroup of $B^\phi_t$.
Let $(B^\phi_t)_+$ be the set of all
$x\in B^\phi_t$ for which there exists
an $\alpha$-regressive trajectory
$(x_{-n})_{n\in\N_0}$
such that $x_{-n}\in B^\phi_t$ for all $n\in\N_0$
and $x_0=x$.
As recalled in Section~\ref{fipro},
also $B^\phi_+$ is a compact subgroup of~$B^\phi_t$.
For each $t\in \,]0,R]$,
we have
\[
(B^\phi_t)_+=W^c(t)W^u_b(t);
\]
moreover,
\[
(B^\phi_t)_-=W^s_a(t)W^c(t)
\]
for all $t\in\,]0,r]$,
by Equations~(68) and (73), respectively,
in \cite[proof of Theorem~8.13]{END}.
Moreover,
\[
W^c(t)=(B^\phi_t)_-\cap (B^\phi_t)_+
\]
is a compact \emph{subgroup} of $B^\phi_t$
for each $t\in \,]0,r]$,
see \cite[Remark~8.14]{END}.
\end{numba}
We shall use the following result concerning local unstable manifolds.
\begin{la}\label{germ-cu}
There exists a
local centre-unstable manifold $N$
for $\alpha$ around~$e$.
Its germ at~$e$ is uniquely
determined.
\end{la}
\begin{proof}
The submanifold $N:=W^cW^u_b$
of $U=W^s_aW^cW^u_b\cong W^s_1\times W^c\times W^u_b$
is a local centre-unstable manifold
for $\alpha$ around~$p$,
as $T_eN=T(W^c)\oplus T(W^u_b)=(T_pM)_{\geq 1}$
and $P:=W^cW^u_r$ is an open $e$-neighbourhood
in $N$
such that $\alpha(P)\sub N$.
If also $N'$ is a local centre-unstable manifold for~$\alpha$
around~$p$,
then there exists an open $p$-neighbourhood
$Q\sub N'$ such that $\alpha(Q)\sub N'$.
By Lemma~\ref{enough-uni},
there exists an open $p$-neighbourhood
$O\sub N'$ such that $\alpha(O)$ is open in $N'$
and
\[
O\sub\alpha(O)\sub U\cap N'.
\]
Hence, for each $x\in O$
we find an $\alpha$-regressive trajectory $(x_{-n})_{n\in\N_0}$
in~$O$ such that $x_0=x$. Since $x_{-n}\in O\sub U$ for each $n\in\N_0$,
we have $x\in U_+=N$.
Thus $O\sub N$. As the inclusion map $\iota\colon O\to N$
is $\K$-analytic and $T_p\iota$
is the identity map of $(T_pM)_{\geq 1}$,
the inverse function theorem
shows that $O$ contains an open $p$-neighbourhood~$W$
such that $W=j(W)$ is open in~$N$ and $\id=j|_W\colon W\to W$
is a $\K$-analytic diffeomorphism.
Thus $N'$ and $N$ induce the same
$\K$-analytic manifold
structure on their joint open subset~$W$,
whence the germs of $N$ and $N'$ at~$p$
coincide.
\end{proof}
\begin{la}\label{con-loc}
We have $\con(\alpha)\cap W^s_a(r)W^c(r)=W_s^a(r)$.
Moreover, $W^u_b(r)$
is the set of all $x\in W^c(r)W^u_b(r)$
admitting an $\alpha$-regressive
trajectory $(x_{-n})_{n\in\N_0}$
in $W^c(r)W^u_b(r)$ such that $x_0=x$ and
$x_{-n}\to e$
for $n\to\infty$.
\end{la}
\begin{proof}
If $x\in W^s_a(r)$, then $\alpha^n(x)\to e$
as $n\to\infty$ and thus $x\in\con(\alpha)$,
by~(\ref{subcontra}).
Now assume that $x\in\con(\alpha)\cap W^s_a(r)W^c(r)$.
Then
$x=yz$ for unique $y\in W^s_a(r)$ and $z\in W^c(r)$.
If we had $z\not=e$, then we could find
$t\in \,]0,r[$
such that $z\not\in W^c(t)$.
Then $\alpha^n(y)\in W^s_a(r)$
for each $n\in \N_0$.
Since $\alpha|_{W^c(r)}\colon W^c(r)\to W^c(r)$
is a bijection  which takes $W^c(t)$ onto itself,
we deduce that $\alpha^n(z)\in W^c(r)\setminus W^c(t)$
for all $n\in \N_0$,
entailing that the group element $\alpha^n(x)=\alpha^n(y)\alpha^n(z)$
is in $B^\phi_r\setminus B^\phi_t$
for each $n\in\N_0$.
Hence $\alpha^n(x)\not\to e$, contradiction.
Thus $z=0$ and thus $x\in W^s_a(r)$.\\[2mm]
By \ref{subu}, each $x\in W^u_b(r)$
has an $\alpha$-regressive trajectory of the
asserted form.
Now let $x\in W^c(r)W^u_b(r)$
and assume there exists an $\alpha$-regressive
trajectory $(x_{-n})_{n\in\N_0}$
in $W^r(r)W^u_b(r)$ such that $x_0=x$ and $x_{-n}\to e$
for $n\to\infty$.
Write $x_{-n}=y_{-n}z_{-n}$
with $y_{-n}\in W^c(r)$ and $z_{-n}\in W^u_b(r)$.
Then $y_{-n}\to e$ and $z_{-n}\to e$ as $n\to\infty$.
For each $n\in\N$, we have
\[
y_{-n+1}z_{-n+1}=x_{-n+1}=\alpha(x_{-n})=\alpha(y_{-n})\alpha(z_{-n})
\]
with $\alpha(y_{-n})\in W^c(r)$ and $\alpha(z_{-n})\in \alpha(W^u_b(r))\sub W^u_b$.
As the product map $W^c\times W^u_b\to W^cW^u_b$
is a bijection, we deduce that $y_{-n+1}=\alpha(y_{-n})$
and $z_{-n+1}=\alpha(z_{-n})$.
If we had $y\not=e$,
we could find $t\in\,]0,r[$ such that $y\not\in W^c(t)$.
There would be some $N\in\N$
such that $x_{-n}\in W^c(t)$ for all $n\geq N$.
Since $\alpha(W^c(t))=W^c(t)$,
this would imply $y=\alpha^N(y_{-N})\in W^c(t)$,
a contradiction. Thus $y=0$
and thus $x=z\in W^u_b(r)$.
\end{proof}
\begin{rem}\label{union-con}
Since $W^u_b(r)\sub \alpha(W^u_b(r))$
by \ref{subu}, $(\alpha^n(W^u_b(r))_{n\in\N}$
is an ascending sequence of
compact subgroups of $\con^-(\alpha)$. Moreover,
\[
\con^-(\alpha)=\bigcup_{n\in\N}\alpha^n(W^u_b(r)).
\]
In fact, for $x\in\con^-(\alpha)$
there exists an $\alpha$-regressive
trajectory $(x_{-n})_{n\in\N}$
such that $x_0=x$ and $x_{-n}\to e$ as $n\to\infty$.
There exists $N\in\N$ such that $x_{-n}\in W^s_a(r)W^c(r)W^u_b(r)=B^\phi_r$
for all $n\geq N$. For each $n\geq N$,
the element $x_{-n}$ of $B^\phi_r$
has the $\alpha$-regressive trajectory $(x_{-n-m})_{m\in\N_0}$
in $B^\phi_r$,
whence $x_{-n}\in (B^\phi_r)_+=W^c(r)W^u_b(r)$.
As $x_{-n-m}\in W^c(r)W^u_b(r)$
and $x_{-n-m}\to e$ as $m\to\infty$,
Lemma~\ref{con-loc} shows that $x_{-n}\in W^u_b(r)$
for each $n\geq N$.
In particular, $x_{-N}\in W^u_b(r)$ and $x=x_0=\alpha^N(x_{-N})\in\alpha^N(W^u_b(r))$.
\end{rem}
\section{Proof of Theorems~\ref{thmA} and \ref{thmB}}
We retain the notation of Section~\ref{locstru}.
As is well known,
a homomorphism $f\colon H\to K$ between $\K$-analytic Lie groups
is $\K$-analytic whenever it is $\K$-analytic
on an open $e$-neighbourhood in~$H$.
Applying this to $f=\id_H$,
we see that two Lie group structures on~$H$
coincide if their germ at $e$ coincides.
The uniqueness of the Lie group
structures in parts~(a), (b), (c), and (d)
of Theorem~\ref{thmA}
therefore follows
from the uniqueness statements
concerning manifold germs
in \ref{uni-germ}, \ref{also-nonhypo},
and \ref{nonhypo2};
the uniqueness statement in~(e)
follows from Lemma~\ref{germ-cu}.
It remains to prove the existence of the asserted
Lie group structures, and that they have the properties
described in Theorem~\ref{thmB}.\\[1mm]
\emph{Contraction groups.}
Since $W_a^s$ is a subgroup of~$G$
and a submanifold, it is a Lie subgroup.
By Lemma~\ref{con-loc}, we have $W^s_a\sub\con(\alpha)$.
If $g\in\con(\alpha)$, then $\{\alpha^n(g)\colon n\in\N_0\}$
is relatively compact, whence
there exists $t\in \,]0,r]$
such that
\[
\alpha^n(g)B^\phi_t\alpha^n(g)^{-1}\sub B^\phi_r
\]
for all $n\in\N_0$.
For each $x\in W^s_a(t)$ and each $n\in\N_0$, we then have
\[
\alpha^n(gxg^{-1})=\alpha^n(g)\alpha^n(x)\alpha^n(g)^{-1}\in B^\phi_r,
\]
as $\alpha^n(x)\in W^s_a(t)$.
As a consequence, $gxg^{-1}\in (B^\phi_r)_-=W^s_a(r)W^c(r)$.
Moreover, $gxg^{-1}\in \con(\alpha)$
as $g\in\con(\alpha)$ and $x\in W^s_a(t)\sub \con(\alpha)$.
Hence $gxg^{-1}\in W^s_a(r)$, by Lemma~\ref{con-loc}.
Being a restriction of the $\K$-analytic conjugation map
$G\to G$, $h\mto ghg^{-1}$,
the map
\[
W^s_a(t)\to W^s_a,\quad x\mto gxg^{-1}
\]
is $\K$-analytic.
By the Local Description of
Lie Group Structures (see Proposition 18
in \cite[Chapter~III,\S1, no.\,9]{Bou}),
we get a unique $\K$-analytic manifold
structure on $\con(\alpha)$
making it a Lie group $\con_*(\alpha)$,
such that $W^s_a$
is an open submanifold.
As $W^s_a$ is a submanifold of~$G$,
$\con_*(\alpha)$
is an immersed Lie subgroup of~$G$.
Now $\alpha(W^s_a)\sub W^s_a$
and $\alpha|_{W^s_a}\colon W^s_a\to W^s_a$
is $\K$-analytic since $\alpha$
is $\K$-analytic and $W^s_a$ is a submanifold
of~$G$.
The restriction $\alpha_s$ of $\alpha$ to
an endomorphism of the subgroup $\con_*(\alpha)$
coincides with $\alpha|_{W^s_a}$ on the open subset
$W^s_a$ of~$\con_*(\alpha)$, whence $\alpha_s$
is $\K$-analytic.
If $g\in \con(\alpha)$,
there exists $N\in\N$ such that
$\alpha^n(g)\in B^\phi_r$ for all
$n\geq N$, whence
$\alpha^n(g)\in (B^\phi_r)_-=W^s_a(r)W^c(r)$.
Since also $\alpha^n(g)\in \con(\alpha)$,
Lemma~\ref{con-loc} shows that $\alpha^n(g)\in W^s_a(r)$.
As $\con_*(\alpha)$ hast $W^s_a$ as an open submanifold
and $\alpha^n(g)\to e$ in $W^s_a$ as $N\leq n\to\infty$,
we see that $(\alpha_s)^n(g)=\alpha^n(g)\to e$ also
in $\con_*(\alpha)$. Thus $\con(\alpha_s)=\con_*(\alpha)$.\\[2mm]
\emph{Anti-contraction groups.}
Being a subgroup of~$G$ and a submanifold,
$W_u^b$ is a Lie subgroup.
By Lemma~\ref{con-loc},
we have $W^u_b\sub\con^-(\alpha)$.
If $g\in\con^-(\alpha)$, then
there exists an $\alpha$-regressive
trajectory $(g_{-n})_{n\in\N_0}$
such that $g_0=g$ and $g_{-n}\to e$
as $n\to\infty$.
Notably,
$\{g_{-n}\colon n\in\N_0\}$
is relatively compact, whence
there exists $t\in \,]0,r]$
such that
\[
g_{-n}B^\phi_t(g_{-n})^{-1}\sub B^\phi_r
\]
for all $n\in\N_0$.
For each $x\in W^u_b(t)$,
there exists an $\alpha$-regressive
trajectory $(x_{-n})_{n\in\N_0}$
in $W^u_b(t)$ such that $x_0=x$ and $x_{-n}\to e$
as $n\to\infty$ (see \ref{subu}).
Then $(g_{-n}x_{-n}(g_{-n})^{-1})_{n\in\N_0}$
is an $\alpha$-regressive trajectory
for $gxg^{-1}$
such that $g_{-n}x_{-n}(g_{-n})^{-1}\to e$
as $n\to\infty$ and
\[
g_{-n}x_{-n}(g_{-n})^{-1}\in g_{-n}B^\phi_t (g_{-n})^{-1}\sub B^\phi_r
\]
for all $n\in\N_0$, whence $g_{-n}x_{-n}(g_{-n})^{-1}\in (B^\phi_r)_+=
W^c(r)W^u_b(r)$ for all $n\in\N_0$
and thus $gxg^{-1}\in W^u_b(r)$,
by Lemma~\ref{con-loc}.
Being a restriction of the $\K$-analytic conjugation map
$G\to G$, $h\mto ghg^{-1}$,
the map
\[
W^u_b(t)\to W^u_b,\quad x\mto gxg^{-1}
\]
is $\K$-analytic.
Using the Local Description of
Lie Group Structures,
we get a unique $\K$-analytic manifold
structure on $\con^-(\alpha)$
making it a Lie group $\con^-_*(\alpha)$,
such that $W^u_b$
is an open submanifold.
As $W^u_b$ is a submanifold of~$G$,
$\con^-_*(\alpha)$
is an immersed Lie subgroup of~$G$.
Now $\alpha(W^u_b(r))\sub W^u_b$
and $\alpha|_{W^u_b(r)}\colon W^u_b(r)\to W^u_b$
is $\K$-analytic since $\alpha$
is $\K$-analytic and $W^u_b$ is a submanifold
of~$G$.
The restriction $\alpha_u$ of $\alpha$ to
an endomorphism of $\con_*^-(\alpha)$
coincides with $\alpha|_{W^u_b(r)}$ on the open subset
$W^u_b(r)$ of~$\con^-_*(\alpha)$, whence $\alpha_u$
is $\K$-analytic.
If $g\in \con^-(\alpha)$,
then there exists an $\alpha$-regressive
trajectory $(g_{-n})_{n\in\N_0}$
such that $g_0=g$ and $g_{-n}\to e$ in~$G$
as $n\to\infty$.
Then $g_{-n}\in \con^-(\alpha)$
for all $n\in\N_0$ and
we have seen in Remark~\ref{union-con}
that there exists an $N\in\N$ such that
$g_{-n}\in W^u_b(r)$ for all $n\geq N$.
As $\con^-_*(\alpha)$ and $G$
induce the same topology on $W^u_b(r)$,
we deduce that $g_{-n}\to e$ also
in $\con^-_*(\alpha)$.
Thus $g\in \con^-(\alpha_u)$.\\[2mm]
\emph{Parabolic subgroups.}
$(B^\phi_r)_-=W_a^s(r)W^c(r)$
is a Lie subgroup of~$G$
and a subset of $\parb(\alpha)$.
If $g\in\parb(\alpha)$, then $\{\alpha^n(g)\colon n\in\N_0\}$
is relatively compact, whence
there exists $t\in \,]0,r]$
such that
\[
\alpha^n(g)B^\phi_t\alpha^n(g)^{-1}\sub B^\phi_r
\]
for all $n\in\N_0$.
For each $x\in W^s_a(t)W^c(t)$,
we have $\alpha^n(x)\in W^s_a(t)W^c(t)$
for all $n\in\N_0$ and thus
\[
\alpha^n(gxg^{-1})=\alpha^n(g)\alpha^n(x)\alpha^n(g)^{-1}
\in \alpha^n(g)B^\phi_t\alpha^n(g)^{-1}\sub B^\phi_r,
\]
whence $gxg^{-1}\in (B^\phi_r)_-=W^s_a(r)W^c(r)$.
Being a restriction of the $\K$-analytic conjugation map
$G\to G$, $h\mto ghg^{-1}$,
the map
\[
W^s_a(t)W^c(t)\to W^s_a(r)W^c(r),\quad x\mto gxg^{-1}
\]
is $\K$-analytic.
By the Local Description of
Lie Group Structures,
we get a unique $\K$-analytic manifold
structure on $\parb(\alpha)$
making it a Lie group $\parb_*(\alpha)$,
such that $W^s_a(r)W^c(r)$
is an open submanifold.
As $W^s_a(r)W^c(r)$ is a submanifold of~$G$,
$\parb_*(\alpha)$
is an immersed Lie subgroup of~$G$.
Now $\alpha(W^s_a(r)W^c(r))\sub W^s_a(r)W^c(r)$
and $\alpha|_{W^s_a(r)W^c(r)}\colon W^s_a(r)W^c(r)\to W^s_a(r)W^c(r)$
is $\K$-analytic.
As a consequence, the
restriction $\alpha_{cs}$ of $\alpha$ to
an endomorphism of the subgroup $\parb_*(\alpha)$
is $\K$-analytic.
If $g\in W^s_a(r)W^c(r)$,
then $\{\alpha^n(g)\colon n\in\N_0\}$
is contained in the compact open subgroup $W^s_a(r)W^c(r)$
of $\parb_*(\alpha)$,
whence $g\in \parb(\alpha_{cs})$.
Thus $W^s_a(r)W^c(r)\sub\parb(\alpha_{cs})$,
entailing that the subgroup
$\parb(\alpha_{cs})$ is an $e$-neighbourhood
and hence open in $\parb_*(\alpha)$.\\[2mm]
\emph{Antiparabolic subgroups.}
$(B^\phi_r)_+=W^c(r)W_u^b(r)$ is a Lie subgroup
of~$G$ and a subgroup of $\parb_-(\alpha)$.
If $g\in\parb^-(\alpha)$, then
there exists an $\alpha$-regressive
trajectory $(g_{-n})_{n\in\N_0}$
such that $g_0=g$ and
$\{g_{-n}\colon n\in\N_0\}$
is relatively compact.
Thus,
there exists $t\in \,]0,r]$
such that
\[
g_{-n}B^\phi_t(g_{-n})^{-1}\sub B^\phi_r
\]
for all $n\in\N_0$.
For each $x\in W^c(t)W^u_b(t)=(B^\phi_t)_+$,
there exists an $\alpha$-regressive
trajectory $(x_{-n})_{n\in\N_0}$
in $W^c(t)W^u_b(t)$ such that $x_0=x$.
Then $(g_{-n}x_{-n}(g_{-n})^{-1})_{n\in\N_0}$
is an $\alpha$-regressive trajectory
for $gxg^{-1}$
such that
\[
g_{-n}x_{-n}(g_{-n})^{-1}\in g_{-n}B^\phi_t (g_{-n})^{-1}\sub B^\phi_r
\]
for all $n\in\N_0$, whence
$gxg^{-1}\in (B^r_\phi)_+=W^c(r)W^u_b(r)$.
Being a restriction of the $\K$-analytic conjugation map
$G\to G$, $h\mto ghg^{-1}$,
the map
\[
W^c(t)W^u_b(t)\to W^c(r)W^u_b(r),\quad x\mto gxg^{-1}
\]
is $\K$-analytic.
Using the Local Description of
Lie Group Structures,
we get a unique $\K$-analytic manifold
structure on $\parb^-(\alpha)$
making it a Lie group $\parb^-_*(\alpha)$,
such that $W^c(r)W^u_b(r)$
is an open submanifold.
As $W^c(r)W^u_b(r)$ is a submanifold of~$G$,
$\parb^-_*(\alpha)$
is an immersed Lie subgroup of~$G$.
Since $\alpha|_{W^u_b(r)}\colon W^u_b(r)\to W^u_b$
is continuous and $W^u_b(r)$ is open in $W^u_b$,
there exists $t\,]0,r]$
such that $\alpha(W^u_b(t))\sub W^u_b(r)$.
Now $\alpha(W^c(t)W^u_b(t))\sub W^c(r)W^u_b(r)$,
entailing that
the restriction $\alpha_{cu}$ of $\alpha$ to
an endomorphism of $\parb_*^-(\alpha)$
is $\K$-analytic.
If $g\in W^c(r)W^u_b(r)=(B^r_\phi)_+$,
then there exists an $\alpha$-regressive
trajectory $(g_{-n})_{n\in\N_0}$ in $B^\phi_r$
such that $g_0=g$.
For each $n\in\N_0$,
the sequence $(g_{-n-m})_{m\in\N_0}$
is an $\alpha$-regressive trajectory in $B^r_\phi$
for $g_{-n}$, whence $g_{-n}\in (B^r_\phi)_+=W^c(r)W^u_b(r)$.
As $\parb^-_*(\alpha)$ and~$G$
induce the same topology on $W^c(r)W^b_u(r)$,
we deduce that $g\in \parb^-(\alpha_{cu})$.
Thus $W^c(r)W^u_b(r)\sub\parb^-(\alpha_{cu})$,
showing that the latter is an open subgroup
of $\parb^-_*(\alpha)$.\\[2mm]
\emph{Levi subgroups.}
$W^c(r)$ is a Lie subgroup
of~$G$ and a subgroup of $\lev(\alpha)$.
If $g\in\lev(\alpha)=\parb(\alpha)\cap\parb^-(\alpha)$,
our discussion of $\parb_*(\alpha)$ and
$\parb^-_*(\alpha)$
yield a $t\in\,]0,r]$ such that
$g W^s_a(t)W^c(t)g^{-1}\sub W^s_a(r)W^c(r)$
and $gW^c(t)W^u_b(t)g^{-1}\sub W^c(r)W^u_b(r)$,
whence
\[
gW^c(t)g^{-1}\sub W^s_a(r)W^c(r)\cap W^c(r)W^u_b(r)=W^c(r).
\]
Being a restriction of the $\K$-analytic conjugation map
$G\to G$, $h\mto ghg^{-1}$,
the map
\[
W^c(t)\to W^c(r),\quad x\mto gxg^{-1}
\]
is $\K$-analytic.
Using the Local Description of
Lie Group Structures,
we get a unique $\K$-analytic manifold
structure on $\lev(\alpha)$
making it a Lie group $\lev_*(\alpha)$,
such that $W^c(r)$
is an open submanifold.
As $W^c(r)$ is a submanifold of~$G$,
$\lev_*(\alpha)$
is an immersed Lie subgroup of~$G$.
Now $\alpha(W^c(r))=W^c(r)$,
entailing that
the restriction $\alpha_c$ of $\alpha$ to
an endomorphism of $\lev_*(\alpha)$
is $\K$-analytic.
If $g\in W^c(r)$,
then there exists an $\alpha$-regressive
trajectory $(g_{-n})_{n\in\N_0}$ in $W^c(r)$
such that $g_0=g$. Moreover, $\alpha^n(g)\in W^c(r)$
for all $n\in\N_0$.
As $\lev_*(\alpha)$ and~$G$
induce the same topology on $W^c(r)$,
we deduce that $g\in \parb^-(\alpha_c)$
and $g\in\parb(\alpha_c)$.
Thus $g\in\lev(\alpha_c)$,
and
thus $W^c(r)\sub\lev(\alpha_c)$,
showing that the latter is an open subgroup
of $\lev_*(\alpha)$. $\,\square$
\section{Proof of Theorem~\ref{thmC}}
We start with a lemma.
\begin{la}\label{inclu}
If $G$ is a $\K$-analytic Lie group
over a totally disconnected local field~$\K$
and $\alpha\colon G\to G$ a $\K$-analytic endomorphism,
then the inclusion maps
\[
\lev_*(\alpha)\to\parb_*(\alpha)\quad \mbox{and}\quad
\lev_*(\alpha)\to\parb^-_*(\alpha)
\]
are $\K$-analytic group homomorphisms
and immersions. The actions
\[
\parb_*(\alpha)\times\con_*(\alpha)\to\con_*(\alpha)\quad\mbox{and}\quad
\parb^-_*(\alpha)\times \con^-_*(\alpha)\to\con^-_*(\alpha)
\]
given by $(g,x)\mto gxg^{-1}$ are $\K$-analytic.
\end{la}
\begin{proof}
The first assertion follows from the fact
that $W^c(r)$, $W^s_a(r)W^c(r)$,
and $W^c(r)W^u_b(r)$ are open $e$-neighbourhoods
in $\lev_*(\alpha)$, $\parb_*(\alpha)$ and $\parb^-_*(\alpha)$,
respectively, and the product maps $W^s_a(r)\times W^c(r)\to W^s_a(r)W^c(r)$
and $W^c(r)\times W^u_b(r)\to W^c(r)W^u_b(r)$ are $\K$-analytic
diffeomorphisms.\\[2mm]
To see that the action of $\parb_*(\alpha)$
on $\con_*(\alpha)$ is $\K$-analytic, we verify the hypotheses of
Lemma~\ref{act-ana}.
For each $g\in\parb_*(\alpha)$,
the set $\{\alpha^n(g)\colon n\in\N_0\}$
is relatively compact in~$G$, whence we find
$t\in\,]0,r]$ such that $\alpha^n(g)B^\phi_t\alpha^n(g)^{-1}\sub B^\phi_r$.
For $x\in W^s_a(t)$, we have $\alpha^n(x)\in W^s_a(t)\sub B^\phi_t$
for each $n\in\N_0$, whence $\alpha^n(gxg^{-1})\in B^\phi_r$
and thus $gxg^{-1}\in (B^\phi_r)_-=W^s_a(r)W^c(r)$.
As
\begin{equation}\label{again-2-id}
\alpha^n(gxg^{-1})=\alpha^n(g)\alpha^n(x)\alpha^n(g)^{-1}\to e
\end{equation}
by \ref{comp-to-id}, Lemma~\ref{con-loc} shows that
$gxg^{-1}\in W^s_a(r)$.
As the map $W^s_a(t)\to W^s_a(r)$, $x\mto gxg^{-1}$
is $\K$-analytic, so ist $W^s_a(t)\to\con_*(\alpha)$,
$x\mto gxg^{-1}$.
Next, we note that if $g\in W^s_a(r)W^c(r)$ and $x\in W^c(r)$,
then $gxg^{-1}\in B^\phi_r$
and
\[
\alpha^n(gxg^{-1})=\alpha^n(g)\alpha^n(x)\alpha^n(g)^{-1}\in B^\phi_r
\]
for all $n\in\N_0$
as $W^s_a(r)W^c(r)$ and $W^s_a(r)$ are $\alpha$-invariant.
Hence $gxg^{-1}\in (B^\phi_r)_-=W^a_s(r)W^c(r)$.
Moreover, (\ref{again-2-id}) holds
by \ref{comp-to-id}, whence $gxg^{-1}\in W^s_a(r)$
by Lemma~\ref{con-loc}.
Thus
\[
W^s_a(r)W^c(r)\times W^s_a(r)\to W^s_a(r)\sub\con_*(\alpha),\quad
(g,x)\mto gxg^{-1}
\]
is $\K$-analytic. Finally, let
$x\in \con_*(\alpha)$ be arbitrary.
There exists $N\in\N$ such that $\alpha^n(x)\in B^\phi_r$
for all $n\geq N$.
There exists $t\in\,]0,r]$ such that
\[
\alpha^n(gxg^{-1}x^{-1})\in B^\phi_r\quad
\mbox{for all $n\in \{0,\ldots, N-1\}$
and $g\in W^s_a(t)W^c(t)$.}
\]
Let $g\in W^s_a(t)W^c(t)$. For all $n\geq N$ we have
$\alpha^n(g)\in W^s_a(t)W^c(t)\sub B^\phi_t\sub B^\phi_r$
and $\alpha^n(x)\in B^\phi_r$, whence
\[
\alpha^n(gxg^{-1}x^{-1})=\alpha^n(g)\alpha^n(x)
\alpha^n(g)^{-1}\alpha^n(x)^{-1}\in B^\phi_r.
\]
Hence $gxgx^{-1}\in (B^\phi_r)_-$.
Since $\alpha^n(x)\to e$ holds
and (\ref{again-2-id}), we deduce that
\[
\alpha^n(gxg^{-1}x^{-1})=\alpha^n(gxg^{-1})\alpha^n(x)^{-1}\to e
\]
as $n\to \infty$. Thus $gxg^{-1}x^{-1}\in W^s_a(r)$,
by Lemma~\ref{con-loc}, entailing that the map
\[
W^s_a(t)W^c(t)\to W^s_a(r)\sub\con_*(\alpha),\quad g\mto gxg^{-1}x^{-1}
\]
is $\K$-analytic and hence also the map
$W^s_a(t)W^c(t)\to\con_*(\alpha)$, $g\mto gxg^{-1}$;
all hypotheses of Lemma~\ref{act-ana}
are verified.\\[2mm]
To see that the conjugation action
of $\parb^-_*(\alpha)$ on $\con^-_*(\alpha)$ is $\K$-analytic,
again we verify the hypotheses of Lemma~\ref{act-ana}.
For each $g\in\parb^-_*(\alpha)$,
there exists an $\alpha$-regressive trajectory
$(g_{-n})_{n\in\N_0}$
with $g_0=g$ such that
$\{g_{-n}\colon n\in\N_0\}$
is relatively compact in~$G$.
There exists
$t\in\,]0,r]$ such that $g_{-n}B^\phi_t (g_{-n})^{-1} \sub B^\phi_r$.
For each $x\in W^u_b(t)$,
there exists an $\alpha$-regressive trajectory $(x_{-n})_{n\in\N_0}$
in $W^u_b(t)$ such that $x_0=x$.
Then $g_{-n}x_{-n}(g_{-n})^{-1}
\in B^\phi_r$ for each $n\in\N_0$.
As $(g_{-n-m}x_{-n-m}(g_{-n-m})^{-1})_{m\in\N_0}$
is
an $\alpha$-regressive trajectory for $g_{-n}x_{-n}(g_{-n})^{-1}$,
we conclude that $g_{-n}x_{-n}(g_{-n})^{-1}\in (B^\phi_r)_+=W^c(r)W^u_b(r)$
for each $n\in\N$. By \ref{comp-to-id}, we have
\begin{equation}\label{again-3-id}
g_{-n}x_{-n}(g_{-n})^{-1}\to e\;\,\mbox{as $\,n\to\infty$.}
\end{equation}
Thus Lemma~\ref{con-loc} shows that
$gxg^{-1}\in W^u_b(r)$.
As a consequence, the map $W^u_b(t)\to W^u_b(r)\sub\con^-_*(\alpha)$,
$x\mto gxg^{-1}$
is $\K$-analytic.
Next, for $g\in W^c(r)W^u_b(r)=(B^\phi_r)_+$ and $x\in W^c(r)$,
there exists an $\alpha$-regressive trajectory $(g_{-n})_{n\in\N_0}$
in $B^r_\phi$ with $g_0=g$.
Moreover, there is an $\alpha$-regressive
trajectory $(x_{-n})_{n\in\N_0}$
in $W^u_b(r)W^c(r)$ such that $x_0=x$ and $x_{-n}\to e$
as $n\to\infty$
(see Lemma~\ref{con-loc}).
Then
\[
g_{-n}x_{-n}(g_{-n})^{-1}\in B^\phi_r
\]
for all $n\in\N_0$.
Hence $gxg^{-1}\in (B^\phi_r)_+=W^c(r)W^u_b(r)$.
Moreover, (\ref{again-3-id}) holds
by \ref{comp-to-id}, whence $gxg^{-1}\in W^u_b(r)$
by Lemma~\ref{con-loc}.
Thus
\[
W^c(r)W^u_b(r)\times W^u_b(r)\to W^u_b(r)\sub\con_*^-(\alpha),\quad
(g,x)\mto gxg^{-1}
\]
is $\K$-analytic.
Finally, let
$x\in \con_*^-(\alpha)$ be arbitrary
and $(x_{-n})_{n\in\N_0}$
be an $\alpha$-regressive trajectory
such that $x_0=x$ and $x_{-n}\to e$ as $n\to\infty$.
There exists $N\in\N$ such that $x_{-n}\in B^\phi_r$
for all $n\geq N$.
There exists $t\in\,]0,r]$ such that
\[
gx_{-n}g^{-1}(x_{-n})^{-1}\in B^\phi_r\quad
\mbox{for all $n\in \{0,\ldots, N-1\}$
and $g\in B^\phi_t$.}
\]
Let $g\in W^c(t)W^u_b(t)=(B^\phi_t)_+$
and $(g_{-n})_{n\in\N_0}$
be an $\alpha$-regressive trajectory in $B^\phi_t$
such that $g_0=g$.
For all $n\geq N$ we have $g_{-n},x_{-n}\in B^\phi_r$
and thus
\[
g_{-n}x_{-n}(g_{-n})^{-1}(x_{-n})^{-1}\in B^\phi_r.
\]
Define $y_{-n}:=
g_{-n}x_{-n}(g_{-n})^{-1}(x_{-n})^{-1}$
for $n\in\N_0$.
For each $n\in\N_0$,
the sequence $(y_{-n-m})_{m\in\N_0}$
is an $\alpha$-regressive trajectory for
$y_n$ in $B^\phi_r$
and thus $y_n\in (B^\phi_r)_+$.
Notably, $(y_{-n})_{n\in\N_0}$
is an $\alpha$-regressive trajectory in $(B^\phi_r)_+=W^c(r)W^u_b(r)$.
Since $x_{-n}\to e$, using
(\ref{again-3-id}), we deduce that
\[
y_n=(g_{-n}x_{-n}(g_{-n})^{-1})(x_{-n})^{-1}=\alpha^n(gxg^{-1})\to e
\]
as $n\to \infty$. Hence $gxg^{-1}x^{-1}=y_0\in W^u_b(r)$,
by Lemma~\ref{con-loc}.
As a consequence, the map
\[
W^c(t)W^u_b(t)\to W^u_b(r)\sub\con_*^-(\alpha),\quad g\mto gxg^{-1}x^{-1}
\]
is $\K$-analytic and hence also the map
$W^c(t)W^u_b(t)\to\con_*^-(\alpha)$, $g\mto gxg^{-1}$;
all hypotheses of Lemma~\ref{act-ana}
are verified.
\end{proof}
{\bf Proof of Theorem~\ref{thmC}.}
(a) In the Lie group
\[
H:=\con_*(\alpha)\times \lev_*(\alpha)\times
(\con^-_*(\alpha)^{\op}),
\]
the subset $W^s_a(r)\times W^c(r)\times W^u_b(r)$
is an open identity neighbourhood
and the restriction of $\pi$ to this
set is a $\K$-analytic diffeomorphism
onto the open identity neighbourhood
$B^r_\phi=W^s_a(r)W^c(r)W^u_b(r)$
of~$G$, see~(\ref{localprod}).
Note that
\[
H\times G\to G,\quad ((x,y,z),g)\mto (x,y,z).g:=xygz
\]
is a $\K$-analytic left action of~$H$ on~$G$.
Moreover, $H.e=\Omega$ is open in~$G$
and $\pi\colon H\to \Omega$
is the orbit map~$\sigma^e$.
Since $\pi$ is \'{e}tale at $e$
by the preceding, $\pi$ is \'{e}tale
by Lemma~\ref{fact-2}\,(d).\\[2mm]
(b) By \cite[Lemma~13.1\,(d)]{BGT},
we have $\parb(\alpha)=\con(\alpha)\lev(\alpha)$.
Lemma~\ref{inclu}
entails that the conjugation action of $\lev_*(\alpha)$
on $\con_*(\alpha)$ is $\K$-analytic.
As the conjugation action is used to define the semi-direct
product $\con_*(\alpha)\rtimes \lev_*(\alpha)$,
the latter is a $\K$-analytic Lie group and the product map
\[
p\colon \con_*(\alpha)\rtimes \lev(\alpha)\to \con_*(\alpha)\lev_*(\alpha)=\parb_*(\alpha),\;\;
(x,y)\mto xy
\]
is a group homomorphism.
Being the pointwise product of the
projections onto $x$ and $y$,
the map $p$ is $\K$-analytic.
The restriction of~$p$
to a map
\[
W^s_a(r)\times W^c(r)\to W^s_a(r)W^c(r)
\]
is a $\K$-analytic diffeomorphism onto the open subset
$W^s_a(r)W^c(r)$ or $\parb_*(\alpha)$
(cf.\ (\ref{localprod})).
As a consequence, the group
homomorphism $p$ is \'{e}tale.\\[1mm]
(c) By \cite[Lemma~13.1\,(e)]{BGT},
we have $\parb^-(\alpha)=\con^-(\alpha)\lev(\alpha)$.
Lemma~\ref{inclu}
entails that the conjugation action of $\lev_*(\alpha)$
on $\con_*^-(\alpha)$ is $\K$-analytic.
As the conjugation action is used to define the semi-direct
product $\con^-_*(\alpha)\rtimes \lev_*(\alpha)$,
the latter is a $\K$-analytic Lie group and the product map
\[
p\colon \con^-_*(\alpha)\rtimes \lev(\alpha)\to
\con^-_*(\alpha)\lev_*(\alpha)=\parb^-_*(\alpha),\;\;
(x,y)\mto xy
\]
is a group homomorphism.
Being the pointwise product of the
projections onto $x$ and $y$,
the map $p$ is $\K$-analytic.
We know that the map
\[
q\colon W^c(r)\times W^u_b(r)\to W^c(r)W^u_b(r)=(B^\phi_r)_+,\quad (a,b)\mto ab
\]
is a $\K$-analytic diffeomorphism (cf.\ (\ref{localprod})).
Since $W^c(r)$, $W^u_b(r)$, and $(B^\phi_r)_+$
are subgroups,
we have
\[
(B^\phi_r)_+=((B^\phi_r)_+)^{-1}=W^u_b(r)^{-1}W^c(r)^{-1}=W^ub(r)W^c(r).
\]
The restriction of $p$ to the open set
$W^u_b(r)\times W^c(r)$ has open image
\[
W^u_b(r)W^c(r)=(B^\phi_r)_+
\]
and is given by
$p(x,y)=q(y^{-1},x^{-1})^{-1}$, whence it is
a $\K$-analytic diffeomorphism
onto its open image. As a consequence, the group
homomorphism $p$ is \'{e}tale. $\,\square$
\section{Proof of Theorem~\ref{thmD}}
(a) If $\alpha$ is \'{e}tale,
then $L(\alpha_s)=L(\alpha)|_{\cg_{<a}}$
is an automorphism of $L(\con_*(\alpha))=\cg_{<a}$,
whence $\alpha_s$ is injective on some
$e$-neighbourhood, by the inverse function theorem.
Hence $\alpha|_{W^s_a(t)}$ is injective for some $t\in\,]0,r]$.
Since $\alpha(W^s_a(t))\sub W^s_a(t)$,
we deduce that $(\alpha^n)|_{W^s_a(t)}=(\alpha|_{W^s_a(t)})^n$
is injective, whence $\ik(\alpha)\cap W^s_a(t)=\{e\}$.
Since $W^s_a(t)$ is an open $e$-neighbourhood in $\con_*(\alpha)$,
we deduce that the subgroup $\ik(\alpha)\sub \con_*(\alpha)$
is discrete. If $\car(\K)=0$ and $\alpha$ is not \'{e}tale,
then $\cg_0\supseteq \ker(L(\alpha))\not=\{0\}$,
whence $L(\alpha_s)=L(\alpha)|_{\cg_0}$
is not injective. Since $\car(\K)=0$,
we have $L(\ker(\alpha_s))=\ker L(\alpha_s)$
(compare 4) in \S2 of \cite[Part~II, Chapter~V]{Ser}
and Corollary~1 in \cite[Part~II, Chapter~III, \S10]{Ser}).
As a consequence, $\ker(\alpha_s)$ is
not discrete, whence also the subgroup $\ik(\alpha)$
of $\con_*(\alpha)$
which contains $\ker(\alpha_s)$
is not discrete.\\[2mm]
(b) We give $Q:=\con_*(\alpha)/\ik(\alpha)$
the unique Lie group structure turning the canonical quotient map
$\con_*(\alpha)\to Q$ into an \'{e}tale
$\K$-analytic map. Then $\wb{\alpha_s}\colon Q\to Q$,
$g\ik(\alpha)\to\alpha(g)\ik(\alpha)$
is a well-defined, $\K$-analytic endomorphism
which is contractive as so is $\alpha_s$.
Moreover, $\wb{\alpha_s}$ is injective and \'{e}tale,
whence $\wb{\alpha_s}(Q)$ is an open subgroup of~$Q$
and $\wb{\alpha_s}\colon Q\to\wb{\alpha_s}(Q)$
a $\K$-analytic isomorphism.
We define $H_n:=Q$ for $n\in\N$
and $\phi_{n,m}\colon H_m\to H_n$ via $\phi_{n,m}:=(\wb{\alpha_s})^{n-m}$
for all $n\geq m$ in~$\N$.
Then $((H_n)_{n\in\N},(\phi_{n,m})_{n\geq m})$
is a directed system of $\K$-analytic Lie groups
and $\K$-analytic group homomorphisms
which are \'{e}tale embeddings,
whence the direct limit group
\[
H:=\dl\,H_n
\]
can be given a $\K$-analytic manifold structure making it a $\K$-analytic
Lie group and
each bonding maps $\phi_n\colon H_n\to H$ an \'{e}tale embedding.
Then also $H=\dl\, H_{n+1}$;
if we let $\beta_n\colon H_{n+1}\to H_n$ be the map $\psi_n:=\id_Q$,
we obtain a group homomorphism
\[
\beta:=\dl\, \beta_n\colon \dl\, H_{n+1}\to \dl\, H_n
\]
determined by $\beta\circ \phi_{n+1}=\phi_n$.
As the images of the $\phi_{n+1}$ form
an open cover of~$H$ and $\beta\circ\phi_n$
is $\K$-analytic and \'{e}tale, also $\beta$ is $\K$-analytic
and \'{e}tale.
Moreover, $\beta$ is surjective and injective
and hence a $\K$-analytic automorphism of~$H$.
If we identify $Q$ with an open subgroup of~$H$
by means of $\phi_1$, then $\beta(\phi_1(x))=\beta(\phi_2(\phi_{12}(x)))=
\phi_1(\phi_{12}(x))=\phi_1(\wb{\alpha_s}(x))$
for $x\in Q$ shows that $\beta$ restricts
to $\wb{\alpha_s}$ on~$Q$.
Since $\beta$ is a contractive $\K$-analytic
automorphism, $H$ is nilpotent (see \cite{CON}).
Hence also~$Q$ is nilpotent.
Since $\ik(\alpha)$ is discrete,
there exists a compact open subgroup $U\sub \con_*(\alpha)$
such that $U\cap \ik(\alpha)=\{e\}$.
Then $q|_U$ is injective, whence $U$
is nilpotent.\\[2mm]
(c) Since $\ker(\alpha_s)\sub \ker((\alpha_s)^2)\sub\cdots$
and $\car(\K)=0$,
the union $\ik(\alpha)=\bigcup_{n\in\N}\ker((\alpha_s)^n)$
is closed in~$G$
(see \cite[Proposition~4.20]{AUT}).
By a Baire argument (or \cite[Proposition~1.19]{AUT}),
there exists $n\in\N$ such that $\ker((\alpha_s)^n)$
is open in $\ik(\alpha)$.
Since $\ker((\alpha_s)^n)$
is a Lie subgroup of $\con_*(\alpha_s)$
(as a special case of
Theorems 2 and 3 in \cite[Part~II, Chapter~IV, \S5]{Ser}),
we deduce that also $\ik(\alpha)$
is a Lie subgroup. Using Theorem~1
from loc.\,cit.,
we get a unique $\K$-analytic manifold structure
on $Q:=\con_*(\alpha)/\ik(\alpha)$
turning the canonical quotient map
$q\colon \con_*(\alpha)\to Q$ into a submersion;
by Remark~2) following the cited theorem,
the latter manifold structure
makes $Q$ a $\K$-analytic Lie group.
Then $\wb{\alpha_s}\colon Q\to Q$,
$g\ik(\alpha)\to\alpha(g)\ik(\alpha)$
is a well-defined endomorphism
which is $\K$-analytic as $q$ is a submersion
and $\wb{\alpha_s}\circ q=q\circ \alpha_s$
is $\K$-analytic.
Moreover, $\wb{\alpha_s}$
is contractive as so is $\alpha_s$.
In addition, $\wb{\alpha_s}$ is injective
and hence \'{e}tale, as $\car(\K)=0$
(so that we can use the naturality of
the exponential function
to see that $L(\wb{\alpha}_s)$
is injective and hence an automorphism).
Hence $\wb{\alpha_s}(Q)$ is an open subgroup of~$Q$
and $\wb{\alpha_s}\colon Q\to\wb{\alpha_s}(Q)$
a $\K$-analytic isomorphism.
We define $H_n:=Q$ for $n\in\N$
and $\phi_{n,m}\colon H_m\to H_n$ via $\phi_{n,m}:=(\wb{\alpha_s})^{n-m}$
for all $n\geq m$ in~$\N$.
As in~(b), we obtain a $\K$-analytic Lie group
structure on $H:=\dl\,H_n$
and a contractive $\K$-analytic
automorphism of~$H$
which extends $\wb{\alpha}_s$.
As $H$ admits a contractive $\K$-analytic automorphism,
it is nilpotent.\\[2mm]
(d) Recall that $W^u_b$ is an open submanifold of
$\con^-_*(\alpha)$ and $\alpha(W^u_b(r))\sub W^u_s$.
For each $t\in \,]0,r]$, the group homomorphism
$\alpha|_{W^u_b(t)}$ is an injective
immersion, whence
$\alpha_u(W^u_b(t))$ is an open subgroup of
$W^u_b$ (and thus of $\con^-_*(\alpha)$)
and $\alpha_u|_{W^u_b(t)}$
is a $\K$-analytic diffeomorphism onto
this open subgroup of $W^u_b$ and hence of $\con^-_*(\alpha)$.
As a consequence, $\alpha_u$ is \'{e}tale.
Since $L(\alpha_u)=L(\alpha)|_{\cg_{>b}}$,
the Ultrametric Inverse Function
Theorem (see \cite[Lemma~6.1\,(b)]{IMP})
provides $\theta\in \,]0,r]$
with $b\theta\leq R$
such that
\[
\alpha_u(W^u_b(t))\supseteq W^u_b(bt)\quad
\mbox{for all $t\in \,]0,\theta]$,}
\]
exploiting Remark~\ref{forifthm}.
Thus
\begin{equation}\label{forinve}
\alpha_u(W^u_b(t/b))\supseteq W^u_b(t)\quad\mbox{for all $t\leq b\theta$.}
\end{equation}
Let $S:=(\alpha_u|_{W^u_b(\theta)})^{-1}(W^u_b(b \theta))$.
Then $\alpha_u(S)=W^u_b(b\theta)$,
which is an open subgroup of $\con^-_*(\alpha)$,
and $\alpha_u|_S\colon S\to \alpha_u(S)$
is a $\K$-analytic isomorphism.
Moreover,
\[
(\alpha_u|_S)^{-1}\colon \alpha_u(S)\to S\sub \alpha_u(S)
\]
maps $W^u_b(t)$ into $W^u_b(t/b)$
for each $t\in\,]0,b\theta]$, by (\ref{forinve}),
whence
\[
((\alpha_u|_S)^{-1})^n (W^u_b(b\theta))\sub W^u_b(\theta/b^{n-1})
\]
for each $n\in\N$. As a consequence, $(\alpha_u|_S)^{-1}$
is a contractive endomorphism of $\alpha_u(S)$.
We now define $H_n:=\alpha_u(S)$
for each $n\in\N$.
Using the bonding maps $\phi_{n,m}:=((\alpha_u|_S)^{-1})^{n-m}\colon
H_m\to H_n$, we can form the direct limit group
$H:=\dl\,H_n$ and give it a $\K$-analytic manifold
structure making it a Lie group
and turning each limit map $\phi_n\colon H_n\to H$
into an injective, \'{e}tale, $\K$-analytic
group homomorphism.
As in the proofs of (b) and (c),
we obtain a $\K$-analytic automorphism
$\beta$ of~$H$ which extends $(\alpha_u|_S)^{-1}\colon
\alpha_u(S)\to S\sub \alpha_u(S)$
and is contractive. Notably, $H$ is nilpotent. $\,\square$
\section{Proof of Proposition~\ref{thmE}}
(a) Since $\alpha(W^c(r)=W^c(r)$ and $\alpha|_{W^c(r)}$
is injective (see \ref{nunum}), we see that $(\alpha^n)|_{W^c(r)}
=(\alpha|_{W^c(r)})^n$ is injective for each $n\in\N$,
entailing that $\ik(\alpha)\cap W^c(r)=\{e\}$.
The asssertion follows as $W^c(r)$
is an open $e$-neighbourhood in $\lev_*(\alpha)$.\\[2mm]
(b) As the product map $W^c\times W^u_b\to W^cW^u_b$, $(x,y)\mto xy$
is a bijection,
$\alpha(xy)=\alpha(x)\alpha(y)$
holds for $(x,y)\in W^c(r)\times W^u_b(r)$
and the restrictions of $\alpha$
to mappings
$W^c(r)\to W^c(r)\sub W^c$
and $W^u_b(r)\to W^u_b$
are injective (see \ref{subu}),
we deduce that $\ker(\alpha)\cap W^c(r)W^u_b(r)=\{e\}$.
It remains to recall that $W^c(r)W^u_b(r)$
is an open $e$-neighbourhood in $\parb^-_*(\alpha)$.\\[2mm]
(c) We know from Theorem~\ref{thmD}\,(d) that $\alpha_u$
is \'{e}tale. Hence $\ker(\alpha_u)=\ker(\alpha)\cap\con^-_*(\alpha)$
is discrete in $\con^-_*(\alpha)$. $\,\square$.
\appendix
\section{Proofs for basic facts in Section~\ref{secprel}}\label{appA}
{\bf Proof for \ref{fact-2}.}
(a) Let $y\in (G.x)^o$. For $z\in G.x$,
there exists $g\in G$ such that $gy=z$. Then $g.(G.x)^o$
is an open neighbourhood of~$z$ in~$X$, whence $z\in (G.x)^o$.\\[1mm]
(b) By hypothesis, $V.x$ is an $x$-neighbourhood in~$G.x$
for each $e$-neighbourhood $V\sub G$.
If $g\in G$, then each $g$-neighbourhood contains
$gV$ for some $V$ as before and $\sigma^x(gV)=g.(V.x)$
is a neighbourhood of $g.x$ in~$G.x$.\\[1mm]
(c) If $G_x\sub G$ is the stabilizer of $x$,
then the quotient topology turns the
canonical map $q\colon G\to G/G_x$ into an open map
(see, e.g., \cite[Lemma~6.2\,(a)]{Str}).
Since $G/G_x$ is compact,
the induced continuous bijection $\phi\colon G/G_x\to G.x$
is a homeomorphism. Hence $\sigma^x=\phi\circ q$ is open.\\[1mm]
(d) For each $g\in G$, the left translation
$\lambda_{g^{-1}}\colon G\to G$, $h\mto g^{-1}h$
and the map $\sigma_g:=\sigma(g,\cdot)\colon X\to X$
are $\K$-analytic diffeomorphisms.
Since $\sigma$ is a left action, we have
$\sigma^x=\sigma_g\circ \sigma^x\circ\lambda_{g^{-1}}$.
As the tangent map $T_e(\sigma^x)\colon T_e(G)\to T_x(X)$
is an isomorphism by hypothesis,
also
$T_g(\sigma^x)=T_x(\sigma_g)\circ T_e(\sigma^x)\circ T_g(\lambda_{g^{-1}})$
is an isomorphism. Thus $\sigma^x$ is \'{e}tale at~$g$. $\,\square$\\[2.3mm]
{\bf Proof for \ref{act-ana}.}
For each $g\in G$,
(a) entails that the automorphism $\sigma_g$ of~$H$
is $\K$-analytic. Since also $(\sigma_g)^{-1}=\sigma_{g^{-1}}$
is $\K$-analytic, $\sigma_g$ is a $\K$-analytic automorphism of~$H$.
Let us show that $\sigma$ is $\K$-analytic on an
open neighbourhood of a given point $(g,x)\in G\times H$.
As the left translation $\lambda_g\colon G\to G$,
$z\mto gz$
and the right translation $\rho_x\colon H\to H$, $y\mto yx$ 
are $\K$-analytic diffeomorphisms, it suffices to show
that $\sigma\circ (\lambda_g\times \rho_x)$
is $\K$-analytic on some open $(e,e)$-neighbourhood
in $G\times H$. Let $U$ and $V$ be as in
(d) and $P$ be as in~(c).
For all $(z,y)\in (U\cap P)\times V$, we then have
\[
\sigma(gz,yx)=\sigma_g(\sigma(z,y)\sigma(z,x))
=\sigma_g(\sigma|_{U\times V}(z,y)\sigma^x|_P(z)),
\]
which is a $\K$-analatic function of
$(z,y)\in (U\cap P)\times V$. $\,\square$\\[2.3mm]
{\bf Proof of~\ref{comp-to-id}.}
If $U\sub G$ is an $e$-neighbourhood,
there exists an $e$-neighbourhood $V\sub G$ such that
$g_nVg_n^{-1}\sub U$ for all $n\in\N$.
Now $x_n\in V$ eventually and thus $g_nx_ng_n^{-1}\in U$. $\,\square$\\[2.3mm]
{\bf Proof of Lemma~\ref{enough-uni}.}
Let $P\sub N$ be an open $p$-neighbourhood
such that $f(P)\sub N$.
Then $g:=f|_P\colon P\to N$
is a $\K$-analytic map such that $T_pg=T_pf|_{T_pN}=T_pf|_{(T_pM)_{\geq 1}}$,
which is a linear automorphism of $T_pN$
with
\[
\frac{1}{\|(T_pg)^{-1}\|_{\op}}\geq 1
\]
if we endow $T_pM$ with a norm adapted to $T_pf$
and $F:=T_pN\sub T_pM$ with the induced norm.
Let $\phi\colon U\to V$
be a $\K$-analytic diffeomorphism from
an open $p$-neighbourhood $U\sub N$
onto an open $0$-neighbourhood $V\sub T_pN$,
such that $\phi(p)=0$ and $d\phi|_{T_pN}=\id_{T_pN}$.
After shrinking $U$ and $V$,
we may assume that $U\sub P$
and $f(U)\sub W$.
There exists $r>0$ such that
$B^F_r(0)\sub V$ and $f(\phi^{-1}(B^F_r(0)))\sub U$.
Then
\[
g\colon B_r(0)\to V,\quad x\mto \phi(f(\phi^{-1}(x)))
\]
is a $\K$-analytic map such that $g(0)=0$
and $g'(0)=T_pf|_F$.
By the Ultrametric Inverse Function Theorem
(cf.\
\cite[Lemma~6.1\,(b)]{IMP}),
after shrinking $r$ we may assume that
$g(B_r^F(0))$ is open and
\[
g(B_r^F(0))=g'(0)(B_r^F(0))\supseteq B_r^F(0).
\]
As a consequence, $O:=\phi^{-}(B^F_r(0))$
is an open $p$-neighbourhood in~$N$
such that $f(O)$ is open in~$N$
and $O\sub f(O)\sub U\sub W$. $\,\square$
{\bf Helge  Gl\"{o}ckner}, Universit\"at Paderborn, Institut f\"{u}r Mathematik,\\
Warburger Str.\ 100, 33098 Paderborn, Germany;
{\tt  glockner\at{}math.upb.de}\vfill
\end{document}